\documentclass[reqno]{amsart}
\usepackage{latexsym,amsfonts,amssymb,amsmath,euscript}
\usepackage[latin1]{inputenc} 
\usepackage{graphicx} 
\usepackage{color,fancybox}

\usepackage[all]{xy}
\usepackage{amssymb}
\usepackage{amsmath}
\usepackage{float}
\usepackage{graphics,graphicx}
\usepackage{hyperref}
\usepackage{bm}

\newcommand{\R} {\mathbb{R}}

\newcommand{\eps}{\varepsilon}

\numberwithin{equation}{section}
\newtheorem{teo}{Theorem}[section]
\newtheorem{lem}[teo]{Lemma}
\newtheorem{cor}[teo]{Corolary}
\newtheorem{prop}[teo]{Proposition}

\newtheorem{re}[teo]{Remark}

\begin{document}

\title[Estimates on the distance of inertial manifolds]{Estimates on the distance of inertial manifolds}
\thanks{ Math Subject Classification (2010): 35B42, 35K90}
\author[J. M. Arrieta]{Jos\'{e} M. Arrieta$^{*, \dag}$}
\thanks{$^*$ Corresponding author:  Jos\'e M. Arrieta,  Departamento de Matem\'atica Aplicada, Facultad de Matem\'aticas, 
Universidad Complutense de Madrid, 28040  Madrid, Spain. e-mail: arrieta@mat.ucm.es}
\thanks{$^\dag$ Partially
supported by grants MTM2009-07540 and MTM2012-31298 (MINECO), Spain and Grupo de Investigaci\'on-UCM 920894 ``Comportamiento Asint\'otico y Din\'amica de Ecuaciones Diferenciales-CADEDIF''.}
\address[Jos\'e M. Arrieta]{Departamento de Matem\'atica Aplicada,
Facultad de Ma\-te\-m\'a\-ti\-cas, Universidad Complutense de
Madrid, 28040 Madrid, Spain.} \email{arrieta@mat.ucm.es}

\author[E. Santamar\'ia]{Esperanza Santamar\'ia$^{\dag}$}
\address[Esperanza Santamar\'ia]{Departamento de Matem\'atica Aplicada,
Facultad de Ma\-te\-m\'a\-ti\-cas, Universidad Complutense de
Madrid, 28040 Madrid, Spain.} \email{esperanza.sm@mat.ucm.es}

\date{}

\begin{abstract}
In this paper we obtain estimates on the distance of the inertial manifolds for dynamical systems generated by 
evolutionary parabolic type equations. We consider the situation where the systems are defined in different phase space and we estimate the distance in terms of the distance of the resolvent operators of the corresponding elliptic operators and the distance of the nonlinearities of the equations.
\end{abstract}

\maketitle
\numberwithin{equation}{section}
\newtheorem{theorem}{Theorem}[section]
\newtheorem{lemma}[theorem]{Lemma}
\newtheorem{corollary}[theorem]{Corollary}
\newtheorem{proposition}[theorem]{Proposition}
\newtheorem{remark}[theorem]{Remark}
\allowdisplaybreaks

\maketitle


\section{Introduction}
Many systems coming from Partial Differential Equations of evolutionary type, enjoy the property of having a finite dimensional manifold which is smooth, invariant and exponentially attracting and carries over all the asymptotic dynamic information of the system. All bounded invariant sets (equilibria, periodic orbits, connecting orbits, attractors, etc) lie in this invariant manifold. The 
existence of these manifolds is proved once we guarantee that the associated linear elliptic operator of the system has large enough gaps in the spectrum and it is obtained through an appropriate fixed point argument.    Proving that we have 
these gaps is one of the major difficulties of the theory, but still there is a class of equations (for instance, one dimensional parabolic equations) for which these inertial manifolds exist and once they exist, we can reduce the system to a finite dimensional one, for which more techniques are available. We refer to \cite{Bates-Lu-Zeng1998, Sell&You} for general references on the theory of Inertial manifolds. See also \cite{JamesRobinson} for an accessible introduction to the theory. These inertial manifolds are smooth, see \cite{ChaoLuSell}. We also refer to \cite{Henry1, Hale, B&V2,Sell&You,LibroAlexandre,Cholewa}
for general references on dynamics of evolutionary equations. 

Due to the relevance of these manifolds, the analysis of its behavior under perturbations is very important.  Identifying the kind of perturbations allowed so that the inertial manifold persists and estimating the distance of  the inertial manifolds is an important task which have implications in the analysis of the dynamics of the equations.  One of the first examples in which an analysis of the persistence of inertial manifolds was carried over was in \cite{Hale&Raugel3}, where the 
dynamics of a parabolic equation in a thin domain is analyzed.  This paper has been one of the main motivations for our work. In the case treated in \cite{Hale&Raugel3}, the limit equation is one-dimensional for which the gap condition is satisfied since the elliptic operator is of Sturm-Liouville type and spectral gaps are known to exist.  The inertial manifold of the limiting one-dimensional problem is proved and after an analysis of the continuity of the spectrum under this perturbation, the inertial manifold is lifted to the perturbed 2-dimensional problem in the thin domain. An estimate of the distance of the inertial manifolds is provided, although it is not as sharp as the one we obtain in this paper.  Also, some general results on persistence can be found in \cite{Bates-Lu-Zeng1998}, and  also in \cite{Jones}, where the results are more focused on the numerical approximations of the equations. More recently some results on the behavior of these manifolds under perturbation of the domain  have appeared \cite{Ng2013,Varchon2012}, although they do not provide estimates on the distance of the manifolds. 

In this work we provide estimates on the distance between the inertial manifold of a system and the inertial manifold of a perturbation of it.  The systems may have different phase space (so we may apply these techniques to domain perturbation problems)  and the distance is estimated in terms of two parameters only: the distance of the resolvent operators of the elliptic part and the distance of the nonlinearities of the equations, see Theorem \ref{distaciavariedadesinerciales}.   

We describe now the contents of the paper. 

In Section \ref{setting} we introduce the notation, the main hypothesis that we will impose, {\bf (H1)} related to the convergence of  the resolvent operators and {\bf(H2)} related to the convergence of the nonlinearities.  We also state the main result  of the paper, Theorem \ref{distaciavariedadesinerciales}.

In Section  \ref{linear} we analyze the behavior of the linear part of the equations.  We show the convergence of the spectrum once the resolvent convergence is imposed and obtain different estimates on the linear problems.

In Section \ref{existence} we obtain the existence of the inertial manifolds. To accomplish this task we apply the results from \cite{Sell&You}. 

In Section \ref{distance} using the implicit definition of the inertial manifolds (given as a fixed point of an appropriate functional) 
and with the estimates of Section \ref{linear} we prove the main result. 

We have also included the short Section \ref{comment} with some final comments about the results of the paper.

\section{Setting of the problem and main results}
\label{setting}

Let $A_0$ be a self-adjoint positive linear operator on a separable real Hilbert space $X_0$ with domain $D(A_0)$, that is $A_0:D(A_0)\subset X_0\rightarrow X_0.$  We denote by $X_0^\alpha$, with $\alpha\in[0, 1]$, the fractional power spaces associated to the operator $A_0$ and $\|\cdot \|_{\alpha}$ its norm, defined in the usual way,  see for instance \cite{Henry1, Cholewa}.

We consider the following evolutionary problem,
\begin{equation}\label{problemalimite}
(P_0)\left\{
\begin{array}{r l }
u^0_t+A_0u^0&=F_0(u^0),\\
u^0(0)\in X^\alpha_0,
\end{array}
\right.
\end{equation}
with $F_0: X^\alpha_0\rightarrow X_0$ certain nonlinearity guaranteeing that we have global existence of solutions. 

We also consider a perturbed problem,
\begin{equation}\label{problemaperturbado}
(P_\varepsilon)\left\{
\begin{array}{r l }
u^\varepsilon_t+A_\varepsilon u^\varepsilon&=F_\varepsilon(u^\varepsilon),\qquad 0<\varepsilon\leq \eps_0\\
u^\varepsilon(0)\in X^\alpha_\varepsilon,
\end{array}
\right.
\end{equation}
where $A_\varepsilon$ is also a self-adjoint positive linear operator on a separable real Hilbert space $X_\varepsilon$,  that is 
$A_\varepsilon: D(A_\varepsilon)=X^1_\varepsilon\subset X_\varepsilon\rightarrow X_\varepsilon,$
and the nonlinear term $F_\eps:X_\eps^\alpha\to X_\eps$ is another nonlinearity guaranteeing  also global existence of solutions of \eqref{problemaperturbado}.   We will impose appropriate  hypotheses on $F_\eps$ and $A_\eps$ so 
such that problem $(P_\varepsilon)$ converges to $(P_0)$ as $\varepsilon$ tends to $0$ in some sense. 

Since our aim is to compare different aspects about the asymptotic dynamics of both problems, (\ref{problemalimite}) and (\ref{problemaperturbado}) and these dynamics live in different functional spaces $X_0$, and $X_\eps$, we will 
need to compare functions from $X_0$ and $X_\varepsilon$, ($X_0^\alpha$ and $X_\varepsilon^\alpha$, respectively, with $\alpha\in[0, 1)$ fixed above). We refer to \cite{Car-Pis}  for a general reference where comparison of functions, operators (and their spectrum) defined in different spaces are analyzed, specially for problems related to asymptotic dynamics. See also \cite{Arrieta&Carvalho, dumbel1} for similar approaches to particular perturbation problems.

We assume the existence of linear continuous operators, $E$ and $M$, such that,
$$E: X_0\rightarrow X_\varepsilon,\qquad\textrm{and}\quad M: X_\varepsilon\rightarrow X_0,$$
and,
$$E_{\mid_{X^\alpha_0}}: X_0^\alpha\rightarrow X_\varepsilon^\alpha, \qquad\textrm{and}\quad M_{\mid_{X_\varepsilon^\alpha}}: X_\varepsilon^\alpha\rightarrow X_0^\alpha.$$
Although these operators depend on $\eps$ we will not make explicit this dependence.  We will assume they are bounded uniform in $\varepsilon$ and without loss of generality we will assume
\begin{equation}\label{cotaextensionproyeccion}
\|E\|_{\mathcal{L}(X_0, X_\varepsilon)}, \|M\|_{\mathcal{L}(X_\varepsilon, X_0)}\leq 2,\qquad \|E\|_{\mathcal{L}(X^\alpha_0, X^\alpha_\varepsilon)}, \|M\|_{\mathcal{L}(X^\alpha_\varepsilon, X^\alpha_0)}\leq 2.
\end{equation}
We also assume these operators satisfy the following properties,
\begin{equation}\label{propiedadesextensionproyeccion}
M\circ E= I,\qquad \|Eu_0\|_{X_\varepsilon}\rightarrow \|u_0\|_{X_0}\quad\textrm{for}\quad u_0\in X_0.
\end{equation}
In particular, the first statement in \eqref{propiedadesextensionproyeccion} implies that $E$ is injective and $M$ is surjective.

We will also assume that  the family of operators $A_\varepsilon$, for $0\leq \eps\leq \eps_0$,  have compact resolvent, that is, the resolvent operators are compact for all $\lambda\in\rho(A_\eps)$  where $\rho(A_\varepsilon)$
is the resolvent set of $A_\varepsilon$.
This fact, together with the fact that the operators are selfadjoint,  implies that its spectrum is discrete real and consists only of eigenvalues, each one with finite multiplicity. Moreover, the fact that $A_\varepsilon$, $0\leq \varepsilon\leq \eps_0$, is positive implies that its spectrum is positive. So, we denote by $\sigma(A_\varepsilon)$,  the spectrum of the operator $A_\varepsilon$,  with, 
$$\sigma(A_\varepsilon)=\{\lambda_n^\varepsilon\}_{n=1}^\infty,\qquad\textrm{ and}\quad 0<c\leq\lambda_1^\varepsilon\leq\lambda_2^\varepsilon\leq...\leq\lambda_n^\varepsilon\leq...$$
and we also denote by $\{\varphi_i^\varepsilon\}_{i=1}^\infty$ an associated orthonormal family of eigenfunctions. Observe that the requirement of the operators $A_\eps$ being positive can be relaxed to requiring that they are all bounded from below uniformly in the parameter $\epsilon$. We can always consider the modified operators $A_\eps+cI$  with $c$ a large enough constant to make the modified operators positive.  The nonlinear equations \eqref{problemalimite} and \eqref{problemaperturbado} would have to be rewritten accordingly.  

\par\bigskip 

With respect to the relation between both operators, $A_0$ and $A_\eps$ we will assume the following hypothesis

\vspace{0.5cm}
{\sl \paragraph{\textbf{(H1).}}  With $\alpha$ the exponent from problems (\ref{problemalimite}) and (\ref{problemaperturbado}), we have
\begin{equation}\label{H1equation}
\|A_\varepsilon^{-1}- EA_0^{-1}M\|_{\mathcal{L}(X_\varepsilon, X_\varepsilon^\alpha)}\to 0\quad \hbox{ as } \eps\to 0.
\end{equation}
}
\par\bigskip 
Notice in particular that from \eqref{H1equation} we also have that  $\|A_\varepsilon^{-1}E- EA_0^{-1}\|_{\mathcal{L}(X_0, X_\varepsilon^\alpha)}\to 0$ as $\eps\to 0$.  Let us define $\tau(\eps)$ as an increasing function of $\eps$ such that 
\begin{equation}\label{definition-tau}
\|A_\varepsilon^{-1}E- EA_0^{-1}\|_{\mathcal{L}(X_0, X_\varepsilon^\alpha)}\leq \tau(\eps).
\end{equation}

\par\bigskip

With respect to the nonlinearities $F_0$ and $F_\eps$, 
  \par\bigskip
{\sl 
\paragraph{\textbf{(H2).}} We assume that the nonlinear terms $F_\varepsilon: X^\alpha_\varepsilon\rightarrow X_\varepsilon$ for $0\leq \eps\leq \eps_0$, satisfy:

\begin{enumerate}
\item[(a)]  They are uniformly bounded, that is, there exists a constant $C_F>0$ independent of $\varepsilon$ such that,
$$\|F_\varepsilon\|_{L^\infty(X_\varepsilon^\alpha, X_\varepsilon)}\leq C_F.$$
\item[(b)] They are globally Lipschitz on $X^\alpha_\varepsilon$ with a uniform Lipstichz constant $L_F$, that is,
\begin{equation}\label{LipschitzFepsilon}
\|F_\varepsilon(u)- F_\varepsilon(v)\|_{X_\varepsilon}\leq L_F\|u-v\|_{X_\varepsilon^\alpha}. 
\end{equation}

\item[(c)] They have a uniformly bounded support in $\varepsilon\geq 0$: there exists $R>0$ such that 
$$Supp F_\varepsilon\subset D_{R}=\{u_\varepsilon\in X_\varepsilon^\alpha: \|u_\varepsilon\|_{X_\varepsilon^\alpha}\leq R\}.$$

\item[(d)]  $F_\eps$ approaches $F_0$ in the following sense,
\begin{equation}\label{estimacionefes}
\sup_{u_0\in X^\alpha_0}\|F_\varepsilon (Eu_0)-EF_0(u_0)\|_{X_\varepsilon}=\rho(\varepsilon),
\end{equation}
and $\rho(\varepsilon)\rightarrow 0$ as $\varepsilon\rightarrow 0$.

\end{enumerate} 
}

\par\bigskip  As we will see below, the convergence of the resolvent operators given by hypothesis {\bf (H1)} guarantees the spectral convergence of the operators, that is, the convergence of the eigenvalues and the eigenfunctions (or eigenprojections).  This implies in particular that if we have a gap on the eigenvalues of $A_0$, we will also have, for $\eps$ small enough a similar gap for the eigenvalues of $A_\eps$.  
This fact, together with the uniform estimates on the nonlinerities $F_\eps$ given by hypothesis {\bf(H2)},  guarantees that
we may construct inertial manifolds of the same dimension for all $0\leq \eps\leq \eps_0$.
We will follow the Lyapunov-Perron method, as developed in \cite{Sell&You} to  obtain these inertial manifolds $\mathcal{M}_\varepsilon$, $0\leq\varepsilon\leq\varepsilon_0$. As a matter of fact,  consider $m\in\mathbb{N}$ such that $\lambda_m^0<\lambda_{m+1}^0$ and denote by $\mathbf{P}_{\mathbf{m}}^{\bm\varepsilon}$ the canonical orthogonal projection onto the eigenfunctions, $\{\varphi^\varepsilon_i\}_{i=1}^m$, corresponding to the first $m$ eigenvalues of the operator $A_\varepsilon $, $0\leq\varepsilon\leq\varepsilon_0$ and $\mathbf{Q}^{\bm\varepsilon}_{\mathbf{m}}$ its orthogonal complement, see (\ref{projectionP}) and (\ref{projectionP0}). For technical reasons, we express any element belonging to the linear subspace $\mathbf{P}_{\mathbf{m}}^{\bm\varepsilon}(X_\varepsilon)$ as a linear combination of the elements of the following basis
$$\{\mathbf{P}_{\mathbf{m}}^{\bm\varepsilon}(E\varphi^0_1), \mathbf{P}_{\mathbf{m}}^{\bm\varepsilon}(E\varphi^0_2), ...,\mathbf{P}_{\mathbf{m}}^{\bm\varepsilon}(E\varphi^0_m)\},\qquad\textrm{for}\quad 0\leq\varepsilon\leq\varepsilon_0,$$
with $\{\varphi^0_i\}_{i=1}^m$ the eigenfunctions related to the first $m$ eigenvalues of $A_0$, which will be seen below that is a basis in $\mathbf{P}_{\mathbf{m}}^{\bm\varepsilon}(X_\varepsilon)$ and 
in $\mathbf{P}_{\mathbf{m}}^{\bm\varepsilon}(X_\varepsilon^\alpha)$.  We will denote by $\psi_i^\eps=\mathbf{ P_m^\eps}(E\varphi_i^0)$.

Let us denote by $j_\varepsilon$ the isomorphism from $ \mathbf{P}_{\mathbf{m}}^{\bm\varepsilon}(X_\varepsilon)=[\psi_1^\varepsilon, ..., \psi_m^\varepsilon]$ onto $\mathbb{R}^m$, that gives us the coordinates of each vector. That is,

\begin{equation}\label{definition-jeps}
\begin{array}{rl}
j_\varepsilon:\mathbf{P}_{\mathbf{m}}^{\bm\varepsilon}(X_\varepsilon)&\longrightarrow \mathbb{R}^m, \\
w_\varepsilon&\longmapsto\bar{p},
\end{array}
\end{equation}
where $w_\varepsilon=\sum^m_{i=1} p_i\psi^\varepsilon_i$ and $\bar{p}=(p_1, ..., p_m)$.

We denote by $|\cdot|$ the usual norm in $\mathbb{R}^m$, 
\begin{equation}\label{normacero}
|\bar{p}|=\left(\sum_{i=1}^mp_i^2\right)^{\frac{1}{2}},
\end{equation}
and by $|\cdot|_\alpha$ the following one,
\begin{equation}\label{normaalpha}
|\bar{p}|_\alpha=\left(\sum_{i=1}^mp_i^2(\lambda_i^\varepsilon)^{2\alpha}\right)^{\frac{1}{2}}.
\end{equation}
\vspace{0.4cm}
We consider the spaces $(\mathbb{R}^m, |\cdot|)$ and $(\mathbb{R}^m, |\cdot|_\alpha)$, that is, $\mathbb{R}^m$ with the norm $|\cdot|$ and $|\cdot|_\alpha$, respectively, and notice that for $w_0=\sum^m_{i=1} p_i\psi^0_i$ and $0\leq\alpha<1$ we have that,
\begin{equation}\label{normajepsilon}
\|w_0\|_{X^\alpha_0}=|j_0(w_0)|_\alpha.
\end{equation}

\par\bigskip 
As we mentioned in the introduction, we are looking for inertial manifolds for systems \eqref{problemalimite} and \eqref{problemaperturbado}. That is, finite dimensional manifolds which are smooth, invariant and exponentially attracting and carry over all the asymptotic dynamic information of the systems. These manifolds will be obtained as graphs of appropriate functions.  This motivates the introduction of the set 
$\mathcal{F}_L$,
$$\mathcal{F}_L=\{ \Phi_\varepsilon:\mathbb{R}^m\rightarrow\mathbf{Q}_{\mathbf{m}}^{\bm\varepsilon}(X^\alpha_\varepsilon),\quad\textrm{such that}\quad \textrm{supp } \Phi_\varepsilon\subset B_R\quad \textrm{and}\quad$$
$$\quad \|\Phi_\varepsilon(\bar{p}^1)-\Phi_\varepsilon(\bar{p}^2)\|_{X^\alpha_\varepsilon}\leq L|\bar{p}^1-\bar{p}^2|_\alpha \quad\bar{p}^1,\bar{p}^2\in\mathbb{R}^m \}.$$ 
Then we can show the following result.
\begin{prop}\label{existenciavariedadinercial}
Let hypotheses {\bf (H1)} and {\bf (H2)} be satisfied. Assume also that $m\geq 1$ is such that,
\begin{equation}\label{CondicionAutovaloresFuerte0}
\lambda_{m+1}^0-\lambda_m^0\geq 12L_F\left[(\lambda_m^0)^\alpha+(\lambda_{m+1}^0)^\alpha\right],
\end{equation}
and
\begin{equation}\label{autovalorgrande0}
(\lambda_m^0)^{1-\alpha}\geq 24L_F(1-\alpha)^{-1}.
\end{equation}
Then, there exist $L<1$ and $\varepsilon_0>0$ such that for all $0\leq\varepsilon\leq\varepsilon_0$ there exists an inertial manifold $\mathcal{M}_\varepsilon$ for (\ref{problemalimite}) and (\ref{problemaperturbado}), given by the ``graph'' of a function $\Phi_\varepsilon\in\mathcal{F}_L$. 
\end{prop}

\begin{re} i) Observe that the gap condition is stated for the eigenvalues of the limit problem. In  particular, this implies that the inertial manifold is obtained of the same dimension $m$ for all values of the parameter $0\leq \eps\leq \eps_0$. 
\par\noindent ii) We have written quotations in the word ``graph'' since the manifold $\mathcal{M}_\varepsilon$ is not properly speaking the graph of the function $\Phi_\varepsilon$ but rather the graph of the appropriate function obtained via the isomorphism $j_\eps$ which identifies $\mathbf{P}_{\mathbf{m}}^{\bm\varepsilon}(X_\varepsilon^\alpha)$ with $\R^m$. 
That is,
$$\mathcal{M}_\eps=\{ j_\eps^{-1}(\bar p)+\Phi_\varepsilon(\bar p); \quad \bar p\in \R^m\}$$
\end{re}

The main result we want to show in this article is the following:

\begin{teo}\label{distaciavariedadesinerciales}
Let hypotheses {\bf (H1)} and {\bf (H2)} be satisfied and let $\tau(\eps)$ be defined by  \eqref{definition-tau}. 
Then, under the hypothesis of Proposition \ref{existenciavariedadinercial}, if $\Phi_0$, $\Phi_\varepsilon$ are the maps  that give us the inertial manifolds, then we have,
\begin{equation}\label{distance-inertialmanifolds}
\|\Phi_\varepsilon-E\Phi_0\|_{L^\infty(\mathbb{R}^m, X^\alpha_\varepsilon)}\leq C[\tau(\varepsilon)|\log(\tau(\varepsilon))|+\rho(\varepsilon)],
\end{equation}
with $C$ a constant independent of $\varepsilon$.
\end{teo}
\par\bigskip

\begin{re} Observe that the estimate \eqref{distance-inertialmanifolds} consists of two terms, 
$\tau(\eps)|\log(\tau(\eps))|$, inherited from the distance of the resolvent operators and $\rho(\eps)$ inherited from the distance of the nonlinear terms.  The factor $|\log(\tau(\eps))|$ seems to appear because of technical reasons. A better estimate would be $\|\Phi_\varepsilon-E\Phi_0\|_{L^\infty(\mathbb{R}^m, X^\alpha_\varepsilon)}\leq C[\tau(\varepsilon)+\rho(\varepsilon)],$ which we have not been able to show, although it is very plausible that this would be true and it should be  the optimal rate. 
\end{re}

\section{ Linear analysis and spectral behavior}
\label{linear}
The spectral decomposition of the operator $A_\varepsilon$ implies that if $\lambda\in\rho(A_\varepsilon)$ then,
$$(\lambda-A_\varepsilon)^{-1}u=\sum_{i=1}^\infty\frac{1}{\lambda-\lambda_i^\varepsilon}(u, \varphi_i^\varepsilon)\varphi_i^\varepsilon.$$
In particular, for $\varepsilon\geq 0$,
$$\|(\lambda-A_\varepsilon)^{-1}\|_{\mathcal{L}(X_\varepsilon, X_\varepsilon)}\leq\max_{i\in\mathbb{N}}\left\{\frac{1}{|\lambda-\lambda_i^\varepsilon|},\quad \lambda_i^\varepsilon\in\sigma(A_\varepsilon)\right\}=\frac{1}{dist(\lambda, \sigma(A_\varepsilon))}.$$

For $\alpha\geq0$ and for all $0\leq \varepsilon\leq \eps_0$, let ${A_\varepsilon}_{\mid_{ X^\alpha_\varepsilon}}: X^{1+\alpha}_\varepsilon\subset X^\alpha_\varepsilon\rightarrow X^\alpha_\varepsilon$, with domain $X^{1+\alpha}_\varepsilon\subset X^1_\varepsilon$, be the restriction of $A_\varepsilon$ to the fractional power space $X^\alpha_\varepsilon\subset X_\varepsilon$ so that,
$$A_\varepsilon u= {A_\varepsilon}_{|_{X^\alpha_\varepsilon}} u\qquad\forall u\in X^{1+\alpha}_\varepsilon.$$
Then ${A_\varepsilon}_{|_{X^\alpha_\varepsilon}}$ is also a sectorial operator on $X^\alpha_\varepsilon$ and with a similar spectral decomposition as above, we can also obtain the estimate
$$\|(\lambda I-A_\varepsilon)^{-1}\|_{\mathcal{L}(X_\varepsilon^\alpha, X_\varepsilon^\alpha)}\leq\frac{1}{dist(\lambda, \sigma(A_\varepsilon))},\qquad 0\leq \varepsilon\leq \eps_0.$$

\bigskip

Moreover, since $A_\varepsilon$ is a sectorial operator, $-A_\varepsilon$ is the infinitesimal generator of a linear semigroup that we denote as $e^{-A_\varepsilon t}$, where,
$$e^{-A_\varepsilon t}=\frac{1}{2\pi i}\int_{\Gamma}(\lambda I+A_\varepsilon)^{-1}e^{\lambda t} d\lambda,$$
with $\Gamma$ a contour in the resolvent set of $-A_\varepsilon$, $\rho(-A_\varepsilon)$, with $arg\lambda\rightarrow \pm \theta$ as $|\lambda|\rightarrow \infty$ for some $\theta\in(\frac{\pi}{2}, \pi)$, (see \cite{Henry1}). Since $A_\varepsilon$, $\varepsilon\geq 0$, is a self-adjoint operator,  the formula above is equivalent to
\begin{equation}\label{semigroup-series}
e^{-A_\varepsilon t}u=\sum_{i=1}^\infty e^{-\lambda^\varepsilon_i t}(u, 	\varphi_i^\varepsilon)\varphi_i^\varepsilon.
\end{equation}
Moreover, we have the following result.
\begin{lem}\label{estimacionsemigrupolineal}
We have the following estimates for the linear semigroup
$$\|e^{-A_\varepsilon t}\|_{\mathcal{L}(X_\varepsilon, X_\varepsilon)}\leq e^{-\lambda_1^\varepsilon t}\leq 1,$$
and,
$$\|e^{-A_\varepsilon t}\|_{\mathcal{L}(X_\varepsilon, X_\varepsilon^\alpha)}\leq e^{-\lambda_1^\varepsilon t}\left(\max\{\lambda_1^\varepsilon, \frac{\alpha}{t}\}\right)^\alpha,$$
for $t\geq 0$.
\end{lem}

\begin{proof}  With the expression of the semigroup given by \eqref{semigroup-series}, we get
 $$\|e^{-A_\varepsilon t}u\|_{X_\varepsilon^\alpha}=\left(\sum_{i=1}^\infty e^{-2\lambda_i^\varepsilon t}(u, \varphi_i^\varepsilon)^2(\lambda_i^\varepsilon)^{2\alpha}\right)^{\frac{1}{2}}.$$
The function $f(\lambda)=e^{-\lambda t}\lambda^\alpha$ attains its maximum at $\lambda=\frac{\alpha}{t}$. Then, we have to distinguish two cases:
 \begin{itemize}
 \item[]If $\frac{\alpha}{t}<\lambda_{1}^\varepsilon$, we obtain
 $$\|e^{-A_\varepsilon t}u\|_{X_\varepsilon^\alpha}\leq e^{-\lambda_{1}^\varepsilon t} (\lambda_{1}^\varepsilon)^\alpha \|u\|_{X_\varepsilon}.$$
 \item[]And if $\lambda_{1}^\varepsilon\leq \frac{\alpha}{t}$,
 $$\|e^{-A_\varepsilon t}u\|_{X_\varepsilon^\alpha}\leq e^{-\alpha}\left(\frac{\alpha}{t}\right)^\alpha\|u\|_{X_\varepsilon}\leq e^{-\lambda_{1}^\varepsilon t}\left(\frac{\alpha}{t}\right)^\alpha\|u\|_{X_\varepsilon}.$$
 \end{itemize}
 That is,
 $$\|e^{-A_\varepsilon t}u\|_{X_\varepsilon^\alpha}\leq e^{-\lambda_{1}^\varepsilon t}\left(\max\{\lambda_{1}^\varepsilon, \frac{\alpha}{t}\}\right)^\alpha\|u\|_{X_\varepsilon} .$$
 In the same way, since
 $$\|e^{-A_\varepsilon t}u\|_{X_\varepsilon}=\left(\sum_{i=1}^\infty e^{-2\lambda_i^\varepsilon t}(u, \varphi_i^\varepsilon)^2\right)^{\frac{1}{2}},$$
 then, we obtain,
 $$\|e^{-A_\varepsilon t}u\|_{X_\varepsilon}\leq e^{-\lambda_{1}^\varepsilon t}\|u\|_{X_\varepsilon} .$$
 This concludes the proof of the result.  
 \end{proof}
\par\bigskip

With respect to the relation of the spectrum we have the following result.
\begin{lem}\label{spectrum-behavior}
If $K_0$ is a compact set of the complex plane with $K_0\subset \rho(A_0)$, the resolvent set of $A_0$, and hypothesis {\bf (H1)} is satisfied, then there exists $\eps_0(K_0)>0$ such that $K_0\subset \rho(A_\eps)$ for all $0<\eps\leq \eps_0(K_0)$. Moreover,  we have the estimates:
\begin{equation}\label{uniform-bounds-spectrum}
\|(\lambda I-A_\eps)^{-1}\|_{\mathcal{L}(X_\varepsilon, X^\alpha_\varepsilon)}\leq C(K_0),\qquad \|(\lambda I-A_\eps)^{-1}\|_{\mathcal{L}(X_\varepsilon, X_\varepsilon)}\leq C(K_0),
\end{equation}
for all $\lambda\in K_0$, $0<\eps\leq \eps_0(K_0)$. 
 \end{lem}
 \begin{proof} Let us start by showing the following: if $\lambda_{\varepsilon_n}\in\rho(A_{\varepsilon_n})$ with $\|(\lambda_{\varepsilon_n} I-A_{\varepsilon_n})^{-1}\|_{\mathcal{L}(X_{\varepsilon_n},X_{\varepsilon_n}^\alpha)}\geq k_n$, $k_n\rightarrow +\infty$ as $n\rightarrow +\infty$, and $\lambda_{\varepsilon_n}\rightarrow \lambda_0$, then $\lambda_0\in\sigma(A_0)$.

Then, assume there exists a sequence $\{\lambda_{\varepsilon_n}\}\in \rho(A_{\varepsilon_n})$ with 
$\|(\lambda_{\varepsilon_n}I-A_{\varepsilon_n})^{-1}\|_{\mathcal{L}(X_{\varepsilon_n},X_{\varepsilon_n}^\alpha)}\geq k_n$, and such that $\lambda_{\varepsilon_n}\rightarrow \lambda_0$ as $\varepsilon_n\rightarrow 0$, for some $\lambda_0$. This implies that there exists  $f_{\eps_n}\in X_{\varepsilon_n}$ with $\|f_{\eps_n}\|_{X_{\varepsilon_n}}=1$ and if
$w_{\eps_n}=(\lambda_{\varepsilon_n} I-A_{\eps_n})^{-1}f_{\eps_n}$, then $\|w_{\eps_n}\|_{X_{\varepsilon_n}^\alpha}\to +\infty$.

If we define $u_{\eps_n}=w_{\eps_n}/\|w_{\eps_n}\|_{X^\alpha_{\varepsilon_n}}$, then $\lambda_{\varepsilon_n}u_{\varepsilon_n}-A_{\eps_n} u_{\eps_n}=f_{\eps_n}/\|w_{\eps_n}\|_{X^\alpha_{\varepsilon_n}}$, which implies 
$$ A_{\eps_n} u_{\eps_n}= \lambda_{\eps_n} u_{\eps_n} -\frac{f_{\eps_n}}{\|w_{\eps_n}\|_{X^\alpha_{\varepsilon_n}}}.$$
Let $\hat{u}_{\varepsilon_n}\in X^\alpha_0$ satisfy the following equation,
\begin{equation}\label{hatU}
A_0\hat{u}_{\varepsilon_n}= \lambda_{\eps_n} Mu_{\eps_n} -\frac{M f_{\eps_n}}{\|w_{\eps_n}\|_{X^\alpha_{\varepsilon_n}}}.
\end{equation}
If we study the norm of the right side, since $\left\|\frac{Mf_{\varepsilon_n}}{\|w_{\varepsilon_n}\|_{X^\alpha_{\varepsilon_n}}}\right\|_{X_0}\rightarrow 0$, we have, by (\ref{cotaextensionproyeccion})
$$\left\|\lambda_{\eps_n} Mu_{\eps_n} -\frac{M f_{\eps_n}}{\|w_{\eps_n}\|_{X^\alpha_{\varepsilon_n}}}\right\|_{X_0}\leq2|\lambda_{\varepsilon_n}|\|u_{\varepsilon_n}\|_{X_{\varepsilon_n}}+\left\|\frac{Mf_{\varepsilon_n}}{\|w_{\varepsilon_n}\|_{X^\alpha_{\varepsilon_n}}}\right\|_{X_0}\leq C.$$
So, $\{\hat{u}_{\varepsilon_n}\}\subset X^\alpha_0$ is a compact family. Then, there exists a $\hat{u}_0\in X_0^\alpha$ and a subsequence, we denote it again as $\hat{u}_{\varepsilon_n}$, such that $\hat{u}_{\varepsilon_n}\rightarrow \hat{u}_0$ in $X_0^\alpha$, as $\varepsilon_n\rightarrow 0$.
Moreover, by hypothesis (H1), we have, $\|u_{\eps_n}-E\hat{u}_{\varepsilon_n}\|_{X_{\varepsilon_n}^\alpha}\to 0$. And,
$$\|u_{\varepsilon_n}-E\hat{u}_0\|_{X_{\varepsilon_n}}\leq \|u_{\varepsilon_n}-E\hat{u}_{\varepsilon_n}\|_{X_{\varepsilon_n}}+\|E\hat{u}_{\varepsilon_n}-E\hat{u}_0\|_{X_{\varepsilon_n}}\leq$$
$$\leq \|u_{\varepsilon_n}-E\hat{u}_{\varepsilon_n}\|_{X_{\varepsilon_n}}+2\|\hat{u}_{\varepsilon_n}-\hat{u}_0\|_{X_0}\rightarrow 0.$$
So, again by (\ref{cotaextensionproyeccion}),
$$\|M u_{\varepsilon_n}-\hat{u}_0\|_{X_0}=\|M(u_{\varepsilon_n}-E\hat{u}_0)\|_{X_0}\leq 2\|u_{\varepsilon_n}-E\hat{u}_0\|_{X_{\varepsilon_n}}\rightarrow 0.$$

Hence, via subsequences, $\lambda_{\eps_n} M u_{\eps_n} -\frac{M f_{\eps_n}}{\|w_{\eps_n}\|_{X^\alpha_{\varepsilon_n}}}\to \lambda_0 \hat{u}_0$  in $X_0$ for some
$\hat{u}_0\in X^\alpha_0$. Also, from the definition of $u_{\varepsilon_n}$ we have that $\|u_{\varepsilon_n}\|_{X_{\varepsilon_n}^\alpha}=1$. Then $1=\|u_{\varepsilon_n}\|_{X_{\varepsilon_n}^\alpha}\leq\|u_{\varepsilon_n}-E\hat{u}_0\|_{X_{\varepsilon_n}^\alpha}+\|E\hat{u}_0\|_{X^\alpha_{\varepsilon_n}}\leq\|u_{\varepsilon_n}-E\hat{u}_0\|_{X_{\varepsilon_n}^\alpha}+2\|\hat{u}_0\|_{X_0^\alpha}.$ But since $\|u_{\varepsilon_n}-E\hat{u}_0\|_{X_{\varepsilon_n}^\alpha}\rightarrow 0$ then $\|\hat{u}_0\|_{X_0^\alpha}>0$ and hence $\hat{u}_0\neq 0$. So, from equation (\ref{hatU}) and the above estimates, we obtain $A_0 \hat{u}_0=\lambda_0 \hat{u}_0$, which shows that $\lambda_0\in \sigma(A_0)$.

\bigskip

Next, we apply this result to prove our lemma. For the first part, we proceed as follows. If $K_0\cap \sigma(A_\eps)$ is non empty for $\eps$ small enough, then there exists a sequence $\varepsilon_n\rightarrow 0$ and $\hat{\lambda}_{\varepsilon_n}\in K_0\cap \sigma(A_{\eps_n})$.  Since the spectrum of $A_{\varepsilon_n}$ is discrete for all $\varepsilon_n$, for each $n$ we can choose $\lambda_{\varepsilon_n}\in\rho(A_{\varepsilon_n})$ such that $|\lambda_{\varepsilon_n}-\hat{\lambda}_{\varepsilon_n}|<\frac{1}{n}$ and $\|(\lambda_{\varepsilon_n} I-A_{\varepsilon_n})^{-1}\|_{\mathcal{L}(X_{\varepsilon_n},X^\alpha_{\varepsilon_n})}>k_n$ with $k_n\rightarrow +\infty$. Moreover, since $K_0$ is compact, there is a subsequence $\hat{\lambda}_{\hat{\varepsilon}_n}$  with $\hat{\lambda}_{\hat{\varepsilon}_n}\to \lambda_0$ and $\lambda_0\in K_0$. Then, we have just proved that, $\lambda_0\in \sigma(A_0)$. This is a contradiction. So, $K_0\cap \sigma(A_{\eps})$ is empty, and then $K_0\subset \rho(A_{\varepsilon})$ as we wanted to prove. 
\bigskip

To obtain the desired estimates, suppose there exist  sequences $\{\lambda_n\}\in K_0$  and $\{\varepsilon_n\}$ with $\varepsilon_n\rightarrow 0$ as $n\rightarrow +\infty$ such that,
$$\|(\lambda_n I-A_{\varepsilon_n})^{-1}\|_{\mathcal{L}(X_{\varepsilon_n}, X_{\varepsilon_n}^\alpha)}\geq k_n,$$
with $k_n\rightarrow +\infty$. Since $K_0$ is a compact set, there exists a $\lambda_0\in K_0$ and a subsequence $\{\lambda_{n_k}\}\in K_0$ with $\lambda_{n_k}\rightarrow \lambda_0$, $\lambda_0\in K_0$, and
$$\|(\lambda_{n_k} I-A_{\varepsilon_{n_k}})^{-1}\|_{\mathcal{L}(X_{\varepsilon_{n_k}}, X^\alpha_{\varepsilon_{n_k}})}\geq k_{n_k}.$$

Then, we have proved above that, $\lambda_0\in\sigma(A_0)$. This is a contradiction because $\lambda_0\in K_0\subset \rho(A_0)$. So, we have for $\lambda\in K_0$,
$$\|(\lambda I-A_\eps)^{-1}\|_{\mathcal{L}(X_\varepsilon, X^\alpha_\varepsilon)}\leq C(K_0),\qquad \|(\lambda I-A_\eps)^{-1}\|_{\mathcal{L}(X_\varepsilon, X_\varepsilon)}\leq C(K_0).$$
This concludes the proof. 
\end{proof}

\bigskip

\begin{re}\label{autovalores}
The result just proved implies the uppersemicontinuity of the spectrum: if $\lambda_\eps\in \sigma(A_\eps)$ and $\lambda_\eps \to \lambda_0$ (via subsequences) then $\lambda_0\in \sigma(A_0)$.   
\end{re}

\bigskip

Now we want to estimate $\|(\lambda I+A_\varepsilon)^{-1}E-E(\lambda I+A_0)^{-1}\|_{\mathcal{L}(X_0, X^\alpha_\varepsilon)}$. We have the following result.
\begin{lem}\label{estimacionlandas}
With the notation above and assuming hypothesis {\bf (H1)} is satisfied, if $\lambda\in\rho(-A_0)$ and $\varepsilon$ is small enough so that $\lambda\in\rho(-A_\varepsilon)$,  we have 
$$\|(\lambda I+A_\varepsilon)^{-1}E-E(\lambda I+A_0)^{-1}\|_{\mathcal{L}(X_0, X^\alpha_\varepsilon)}\leq C^\varepsilon_3(\lambda)\tau(\varepsilon),$$
where $C^\varepsilon_3(\lambda)=\left(1+\frac{|\lambda|}{dist(\lambda, \sigma(-A_\varepsilon))}\right)\left(1+\frac{|\lambda|}{dist(\lambda, \sigma(-A_0))}\right)$ and $\tau(\eps)$ is defined by \eqref{definition-tau}.
\end{lem}
\begin{proof}
First of all notice that from Lemma \ref{spectrum-behavior} if $\lambda\in\rho(-A_0)$ then $\lambda\in\rho(-A_\varepsilon)$ for $\varepsilon$ small enough. Hence $(\lambda I+A_\varepsilon)^{-1}$ and $(\lambda I+A_0)^{-1}$ are well defined for all $\lambda\in \rho(-A_0)$. 

We are interested in estimating, $$\|(\lambda I+A_\varepsilon)^{-1}E-E(\lambda I+A_0)^{-1}\|_{\mathcal{L}(X_0,X_\varepsilon^\alpha)}.$$
The first thing we are going to do is to show the following identity:
\begin{equation}\label{identidadprincipal} 
(\lambda I+A_\varepsilon)^{-1}E-E(\lambda I+A_0)^{-1}=  
[I-(\lambda I+A_\varepsilon)^{-1}\lambda ](A_\varepsilon^{-1}E-E A_0^{-1})[I-\lambda(\lambda I+A_0)^{-1}].
\end{equation}
First, note that 
\begin{equation}\label{identidad}
(I+A_\varepsilon^{-1}\lambda)[I-(A_\varepsilon+\lambda I)^{-1}\lambda]= I,
\end{equation}
then,
$$(I+A_\varepsilon^{-1}\lambda)(\lambda I+A_\varepsilon)^{-1}=A_\varepsilon^{-1}.$$
Hence,
$$(I+A_\varepsilon^{-1}\lambda)\left[(\lambda I+A_\varepsilon)^{-1}E-E(\lambda I+A_0)^{-1}\right]=$$
$$=A_\varepsilon^{-1}E-E(\lambda I+A_0)^{-1}-A_\varepsilon^{-1}\lambda E(\lambda I+A_0)^{-1}.$$
Since, 
$$E(\lambda I+A_0)^{-1}=EA_0^{-1}- EA_0^{-1}+E(\lambda I+A_0)^{-1}=EA_0^{-1}- EA_0^{-1}[I-A_0( \lambda I+A_0)^{-1}]=$$
$$=EA_0^{-1}- EA_0^{-1}[(A_0+\lambda I)^{-1}\lambda],$$
we have,
$$(I+A_\varepsilon^{-1}\lambda)\left[(\lambda I+A_\varepsilon)^{-1}E-E(\lambda I+A_0)^{-1}\right]=$$
$$=A_\varepsilon^{-1}E-A_\varepsilon^{-1}E\lambda(\lambda I+A_0)^{-1}-EA_0^{-1}+EA_0^{-1}[(A_0+\lambda I)^{-1}\lambda]=$$
$$=(A_\varepsilon^{-1} E-EA_0^{-1})[I-\lambda(\lambda I+A_0)^{-1}].$$
By (\ref{identidad}), $[I-\lambda(A_\varepsilon+\lambda I)^{-1}](I+A_\varepsilon^{-1}\lambda)= I$, then we obtain the desired identity (\ref{identidadprincipal}),
$$(\lambda I+A_\varepsilon)^{-1}E-E(\lambda I+A_0)^{-1}=[I-\lambda(A_\varepsilon+\lambda I)^{-1}](A_\varepsilon^{-1} E-EA_0^{-1})[I-\lambda(\lambda I+A_0)^{-1}].$$
Hence, since hypothesis {\bf (H1)} is satisfied, we obtain the desired estimates,
$$\|(\lambda I+A_\varepsilon)^{-1}E-E(\lambda I+A_0)^{-1}\|_{\mathcal{L}(X_0, X^\alpha_\varepsilon)}\leq $$
$$\leq\|(I-\lambda(A_\varepsilon+\lambda I)^{-1}\|_{\mathcal{L}(X^\alpha_\varepsilon, X^\alpha_\varepsilon)}\|A_\varepsilon^{-1}E-E A_0^{-1}\|_{\mathcal{L}(X_0, X^\alpha_\varepsilon)}\|I-\lambda(\lambda I+A_0)^{-1}\|_{\mathcal{L}(X_0, X_0)}\leq$$
$$\leq \left(1+\frac{|\lambda|}{dist(\lambda, \sigma(-A_\varepsilon))}\right)\tau(\varepsilon)\left(1+\frac{|\lambda|}{dist(\lambda, \sigma(-A_0))}\right).$$
This concludes the proof.
\end{proof}
\par\bigskip 
We can easily show now,

\begin{cor}\label{C3}
(i) If $K_0\subset\rho(-A_0)$ as in Lemma \ref{spectrum-behavior} and $\Sigma_{-a, \phi}$ is the set of the complex plane described by 
$$\Sigma_{-a, \phi}=\{\lambda\in\mathbb{C}: |arg(\lambda+a)|\leq \pi-\phi\},$$
with $a\geq 0$,  then,
$$\sup_{\lambda\in K_0\cup \Sigma_{-a, \phi}}C^\varepsilon_3(\lambda)\leq {C}_3.$$
(ii) If we take $a=0$ and $\phi=\frac{\pi}{4}$ then
\begin{equation}\label{cotaC3}
C_3^\varepsilon(\lambda)\leq\left(1+\frac{1}{\sin(\phi)}\right)^2\leq 6, \qquad\textrm{for all}\quad \lambda\in \Sigma_{0, \frac{\pi}{4}}.
\end{equation}
\end{cor}
\vspace{0.8cm}

\begin{re}\label{C3acotado}
Note that, although $C_3^\varepsilon(\lambda)$ depends on $\varepsilon$, thanks to the uppersemicontinuity of the eigenvalues, see Remark \ref{autovalores}, we can consider it uniform in $\varepsilon$.
\end{re}

\bigskip

The estimate found in Lemma \ref{estimacionlandas} will be applied to obtain estimates on the distance of the spectral projections and estimates on the distance of the linear semigroups generated by $A_0$ and $A_\eps$.   
Let us start with the spectral projections. 

Let us assume that for some $m=1, 2, ...$ we have $\lambda_m^0<\lambda_{m+1}^0$ and as we have mentioned in the introduction, we denote by $\{\varphi^\varepsilon_i\}_{i=1}^m$  the first $m$ eigenfunctions of the operator $A_\varepsilon $, $0\leq\varepsilon\leq\varepsilon_0$ and by $\mathbf{P}_{\mathbf{m}}^{\bm\varepsilon}$ the canonical orthogonal projection onto the subspace $[\varphi_1^\eps, \ldots, \varphi_m^\eps]$, that is, if $0<\eps\leq \eps_0$
\begin{equation}\label{projectionP}
\begin{array}{rl}
\mathbf{P}_{\mathbf{m}}^{\bm\varepsilon}:X_\varepsilon&\longrightarrow X_\varepsilon \\ \\
v&\longrightarrow \mathbf{P}_{\mathbf{m}}^{\bm\varepsilon}(v)=\sum_{i=1}^m (v,\varphi_i^\eps)_{X_\varepsilon}\varphi_i^\eps 
\end{array}
\end{equation}
or if $\eps=0$, 
\begin{equation}\label{projectionP0}
\begin{array}{rl}
\mathbf{P}_{\mathbf{m}}^{\bm 0}:X_0&\longrightarrow X_0 \\ \\
v&\longrightarrow \mathbf{P}_{\mathbf{m}}^{\bm 0}(v)=\sum_{i=1}^m (v,\varphi_i^0)_{X_0}\varphi_i^0
\end{array}
\end{equation}
Notice that in a natural way, the projections may be defined in the intermediate space $X_\eps^\alpha$ and, since it is a finite linear combination of eigenfunctions, its range is contained also in $X_\eps^\alpha$.

We have the following estimate.
\begin{lem}\label{estimacionProyecciones}
Let $\{\mathbf{P}_{\mathbf{m}}^{\bm\varepsilon}\}_{0\leq\varepsilon\leq\varepsilon_0}$ be the family of canonical orthogonal projections described above, $v\in X_0$, $\Gamma$ a curve in the complex plane contained in $\rho(-A_0)$ and encircling the first $m$ eigenvalues of $-A_0$. Then if we assume {\bf (H1)} is satisfied, we have
$$\|\mathbf{P}_{\mathbf{m}}^{\bm\varepsilon}E(v)-E\mathbf{P}_{\mathbf{m}}^{\mathbf{0}}(v)\|_{X^\alpha_\varepsilon}\leq C_P \tau(\varepsilon)\|v\|_{X_0},$$
with $C_P=\frac{|\Gamma|}{2\pi}\sup_{\lambda\in\Gamma} C^\varepsilon_3(\lambda)$, $|\Gamma|$ the length of the curve $\Gamma$ and $C_3^\eps$ is given in Lemma \ref{estimacionlandas}. 

\end{lem}
\begin{proof}
Let $\Gamma$ be the curve mentioned above. From Lemma \ref{spectrum-behavior},  taking $K_0=\Gamma$, we have that  $\Gamma\subset \rho(-A_\eps)$ for $0\leq \eps\leq \eps_0(\Gamma)$ with $\eps_0(\Gamma)$ small enough.  The spectral projection over the eigenspace generated by the part of the spectrum of $-A_\eps$ contained ``inside'' the curve $\Gamma$ is given by 
$$\mathbf{P}_{\mathbf{\Gamma}}^{\bm\varepsilon}=\frac{1}{2\pi i}\int_{\Gamma} (A_\varepsilon + \lambda I)^{-1}d\lambda,\qquad\textrm{with}\quad \lambda\in\Gamma, \quad 0\leq \eps\leq \eps_0.$$
Therefore, 
$$\|\mathbf{P}_{\mathbf{\Gamma}}^{\bm\varepsilon}E(v)-E\mathbf{P}_{\mathbf{\Gamma}}^{\mathbf{0}}(v)\|_{X^\alpha_\varepsilon}\leq \left|\frac{1}{2\pi i}\right|\left|\int_{\Gamma}\|(\lambda I+A_\varepsilon )^{-1}E(v)-E(\lambda I+A_0 )^{-1}(v)\|_{X^\alpha_\varepsilon}d\lambda\right|.$$
Applying now Lemma \ref{estimacionlandas}, we obtain
\begin{equation}\label{estimate-projectionsGamma}
\|\mathbf{P}_{\mathbf{\Gamma}}^{\bm\varepsilon}E(v)-E\mathbf{P}_{\mathbf{\Gamma}}^{\mathbf{0}}(v)\|_{X^\alpha_\varepsilon}\leq \frac{1}{2\pi}|\Gamma|\sup_{\lambda\in\Gamma}C_3^\eps(\lambda)\tau(\varepsilon)\|v\|_{X_0}= C_P \tau(\varepsilon)\|v\|_{X_0}.
\end{equation}
Since the curve $\Gamma$ encircles only the first $m$ eigenvalues of $-A_0$, then we know that $\mathbf{P}_{\mathbf{\Gamma}}^{\mathbf{0}}=\mathbf{P}_{\mathbf{m}}^{\mathbf{0}}$, that is, the projection over the first $m$ eigenfunctions.   This implies that $Rank(\mathbf{P}_{\mathbf{\Gamma}}^{\mathbf{0}})=m$ and from \eqref{estimate-projectionsGamma}, we also have that $Rank(\mathbf{P}_{\mathbf{\Gamma}}^{\bm\varepsilon})=m$ and therefore we also have $\mathbf{P}_{\mathbf{\Gamma}}^{\bm\varepsilon}=\mathbf{P}_{\mathbf{m}}^{\bm\varepsilon}$. Hence, \eqref{estimate-projectionsGamma} proves the result.  
\end{proof}

\begin{re}\label{continuity}
With a similar argument as the one in the proof of Lemma \ref{estimacionProyecciones}, we may prove the continuity of the eigenvalues and of the spectral projections. If $\lambda_0$ is an eigenvalue of $-A_0$ of multiplicity $s$  and if
$\Gamma=\{ z\in  \mathbb{C}: |z-\lambda_0|=\delta\}$, with $\delta>0$ small enough so that $\sigma(-A_0)\cap \{ z\in  \mathbb{C}: |z-\lambda_0|\leq 2\delta\}=\{ \lambda_0\}$, then for $\eps$ small enough, $\Gamma\subset \rho(-A_\eps)$ and 
$$\|\mathbf{P}_{\mathbf{\Gamma}}^{\bm\varepsilon}E(v)-E\mathbf{P}_{\mathbf{\Gamma}}^{\mathbf{0}}(v)\|_{X^\alpha_\varepsilon}\leq  C\tau(\varepsilon)\|v\|_{X_0}\to 0, \hbox{ as }\eps\to 0,
$$
which implies that the rank of the projection $\mathbf{P}_{\mathbf{\Gamma}}^{\bm\varepsilon}$ is also $s$ and therefore there are exactly $s$ eigenvalues (counting multiplicity) of $-A_\eps$ in $\{ z\in \mathbb{C}: |z-\lambda_0|\leq \delta\}$ and the projections converge. 
\end{re}

\bigskip

We can also obtain good estimates for the linear semigroups.
\begin{lem}\label{estimacionsemigruposlineales}
Let hypothesis {\bf (H1)} be satisfied. If we denote,
$$l^\alpha_\varepsilon(t):=\min\{t^{-1}\tau(\varepsilon),t^{-\alpha}\},\qquad t>0\quad\textrm{and}\quad \alpha\in[0, 1)$$
then, 
\begin{equation}\label{estimacionlinealL2}
\|e^{-A_\varepsilon t}E-E e^{-A_0t}\|_{\mathcal{L}(X_0, X^\alpha_\varepsilon)}\leq 4 l^\alpha_\varepsilon(t).
\end{equation}

\end{lem}
\begin{proof}
Let $\Sigma_{0,\phi}=\{\lambda\in\mathbb{C}: |arg(\lambda)|\leq\pi-\phi\}$, with $\phi=\frac{\pi}{4}$,  and let $\Gamma$ be the boundary of $\Sigma_{0,\frac{\pi}{4}}$, that is the curve consisting of the following segments $\Gamma^1$ and $\Gamma^2$, 
$$\Gamma=\Gamma^1\cup\Gamma^2=\{re^{-i(\pi-\phi)} : 0\leq r<\infty\}\cup\{re^{i(\pi-\phi)}: 0\leq r< +\infty\}$$ oriented such that the imaginary part grows as $\lambda$ runs in $\Gamma$. We know that,
$$e^{-A_\varepsilon t}E-E e^{-A_0t}=\frac{1}{2\pi i}\int_{\Gamma} \left((\lambda I+A_\varepsilon)^{-1}E-E(\lambda I+A_0)^{-1}\right)e^{\lambda t}d\lambda.$$
So,
$$\|e^{-A_\varepsilon t}E-E e^{-A_0t}\|_{\mathcal{L}(X_0, X^\alpha_\varepsilon)}\leq\frac{1}{2\pi}\left|\int_{\Gamma}C_3\tau(\varepsilon)|e^{\lambda t}|d\lambda\right|,$$
with $C_3=\sup_{\lambda\in\Gamma}C_3^\eps(\lambda)$. Since $\lambda\in\Gamma$,
$$|e^{\lambda t}|=|e^{(re^{-i(\pi-\phi)})t}|=e^{(-rcos(\phi))t}\qquad\textrm{for}\quad 0\leq r\leq +\infty,\quad \lambda\in \Gamma^1$$
and,
$$|e^{\lambda t}|=|e^{(re^{i(\pi-\phi)})t}|=e^{(-rcos(\phi))t}\qquad\textrm{for}\quad 0\leq r\leq +\infty,\quad \lambda\in\Gamma^2.$$
With this,

$$\|e^{-A_\varepsilon t}E-Ee^{-A_0t}\|_{\mathcal{L}(X_0, X^\alpha_\varepsilon)}\leq\frac{2}{2\pi}C_3\tau(\varepsilon)\int_0^\infty e^{(-rcos(\phi))t}dr.$$
We make the change of variables $(rcos(\phi))t=z$, and then,
$$\|e^{-A_\varepsilon t}E-E e^{-A_0t}\|_{\mathcal{L}(X_0, X^\alpha_\varepsilon)}\leq\frac{1}{\pi}C_3\tau(\varepsilon)\frac{1}{cos(\phi)t}\int_{0}^\infty e^{-z}dz\leq \frac{1}{\pi cos(\phi)}C_3\tau(\varepsilon)t^{-1},$$
with $C_3=\sup_{\lambda\in\Gamma}C_3^\eps(\lambda)\leq 6$ and, for $\phi=\frac{\pi}{4}$, $\frac{C_3}{\pi cos(\phi)}<4$, which implies 
\begin{equation}\label{eq-aux1}
\|e^{-A_\varepsilon t}E-E e^{-A_0t}\|_{\mathcal{L}(X_0, X^\alpha_\varepsilon)}\leq 4\tau(\eps) t^{-1}.
\end{equation}

\vspace{0.2cm}

On the other hand,
$$\|e^{-A_\varepsilon t}E-Ee^{-A_0 t}\|_{\mathcal{L}(X_0, X^\alpha_\varepsilon)}\leq \|e^{-A_\varepsilon t}E\|_{\mathcal{L}(X_0, X^\alpha_\varepsilon)}+\|E e^{-A_0 t}\|_{\mathcal{L}(X_0, X^\alpha_\varepsilon))}.$$

Then, by Lemma \ref{estimacionsemigrupolineal} and (\ref{cotaextensionproyeccion}),
$$\|e^{-A_\varepsilon t}E-E e^{-A_0 t}\|_{\mathcal{L}(X_0, X^\alpha_\varepsilon)}\leq 2e^{-\lambda_1^\varepsilon t}\left(\max\{\lambda_1^\varepsilon,\frac{\alpha}{t}\}\right)^{\alpha}+2e^{-\lambda_1^0 t}\left(\max\{\lambda_1^0,\frac{\alpha}{t}\}\right)^{\alpha}
$$
But, direct computations show that for each $\lambda>0$ we have $e^{-\lambda  t}\left(\max\{\lambda ,\frac{\alpha}{t}\}\right)^{\alpha}\leq t^{-\alpha}$ and therefore,
\begin{equation}\label{eq-aux2}
\|e^{-A_\varepsilon t}E-E e^{-A_0 t}\|_{\mathcal{L}(X_0, X^\alpha_\varepsilon)}\leq 4t^{-\alpha}
\end{equation}

Putting together \eqref{eq-aux1} and \eqref{eq-aux2}, we get the result. 
\end{proof}

\par\bigskip

For further analysis we will include here some properties of the function $l^\alpha_\varepsilon(t)$ that will be used below.

\begin{lem}\label{lepsilon}
Let $0\leq \gamma<1$ and $a>0$.  If we consider, for all $t>0$,
$$l^\alpha_\varepsilon(t):=\min\{t^{-1}\tau(\varepsilon), t^{-\alpha}\},\qquad\textrm{with}\quad 0\leq\alpha<1,\quad\textrm{and}\quad \tau(\varepsilon)\xrightarrow[\varepsilon\rightarrow 0]{} 0,$$
then, we have the following estimates,
$$\int_0^t (t-s)^{-\gamma}l^\alpha_\varepsilon(s)ds\leq\frac{2^\gamma}{(1-\gamma)(1-\alpha)}t^{-\gamma}(|\log(t))+|\log(\tau(\varepsilon))|)\tau(\varepsilon),$$
$$\int_0^t e^{-as}l^\alpha_\varepsilon(s)ds\leq \frac{2}{1-\alpha}(|\log(t)|+|\log(\tau(\varepsilon))|)\tau(\varepsilon),$$
and,
$$\int_0^\infty e^{-as}l^\alpha_\varepsilon(s)ds\leq \frac{2}{1-\alpha}|\log(\tau(\varepsilon))|\tau(\varepsilon),\qquad\textrm{if}\quad a\geq 1.$$

\end{lem}
\begin{proof}
To prove the first estimate, we divide the analysis in several cases.  First, if $0<t\leq 2\tau(\eps)^{\frac{1}{1-\alpha}}$, we have
$$\int_0^t(t-s)^{-\gamma}l^\alpha_\eps(s)ds\leq \int_0^t(t-s)^{-\gamma}s^{-\alpha}ds=t^{-\gamma+1-\alpha}\int_0^1(1-z)^{-\gamma}z^{-\alpha}dz$$
where we have performed the change of variables $s=tz$ in the integral.  Hence,
$$\int_0^t(t-s)^{-\gamma}l^\alpha_\eps(s)ds\leq Ct^{-\gamma} t^{1-\alpha}\leq Ct^{-\gamma}\tau(\eps).$$

Second, if $2\tau(\eps)^{\frac{1}{1-\alpha}}\leq t$, then 
$$\int_0^t(t-s)^{-\gamma}l^\alpha_\eps(s)ds\leq \int_0^{\tau(\eps)^{\frac{1}{1-\alpha}}}(t-s)^{-\gamma} s^{-\alpha}ds+\int_{\tau(\eps)^{\frac{1}{1-\alpha}}}^{t/2}(t-s)^{-\gamma}s^{-1}\tau(\eps)ds$$
$$+\int_{t/2}^t (t-s)^{-\gamma}s^{-1}\tau(\eps)ds=I_1+I_2+I_3.$$
We study each term separately. For the first one, $I_1$, note that if $t\geq 2\tau(\eps)^{\frac{1}{1-\alpha}}$ and $s\in [0, \tau(\eps)^{\frac{1}{1-\alpha}}]$ then $t-s\geq \frac{t}{2}$. So,
$$I_1\leq \left(\frac{t}{2}\right)^{-\gamma}\int_0^{\tau(\eps)^{\frac{1}{1-\alpha}}}s^{-\alpha}ds\leq 2^\gamma t^{-\gamma}\frac{1}{1-\alpha}\tau(\eps),$$
$$I_2\leq (t/2)^{-\gamma} (\log(t/2)-\log(\tau(\eps)^{\frac{1}{1-\alpha}}))\tau(\eps)\leq 2^\gamma t^{-\gamma}(|\log(t)| +\frac{1}{1-\alpha}|\log(\tau(\eps))|)\tau(\eps),$$
$$I_3\leq t^{-\gamma}\int_{1/2}^1(1-z)^{-\gamma}z^{-1}dz \tau(\eps)\leq \frac{2^\gamma}{1-\gamma}t^{-\gamma}\tau(\varepsilon)\leq\frac{2^\gamma}{1-\gamma}\frac{1}{1-\alpha}t^{-\gamma}\tau(\varepsilon).$$
Putting together the three estimates we show the desired estimate,
$$\int_0^t (t-s)^{-\gamma}l^\alpha_\varepsilon(s)ds\leq \frac{2^\gamma}{(1-\gamma)(1-\alpha)}t^{-\gamma}(|\log(t)|+|\log(\tau(\varepsilon))|)\tau(\varepsilon).$$
For the second estimate, we proceed as follows,
$$\int_0^t e^{-as}l^\alpha_\varepsilon(s)ds=\int_0^{\tau(\varepsilon)^{\frac{1}{1-\alpha}}}e^{-as}s^{-\alpha}ds+\int_{\tau(\varepsilon)^{\frac{1}{1-\alpha}}}^t e^{-as}s^{-1}\tau(\varepsilon)\leq $$
$$\leq\frac{1}{1-\alpha}\tau(\varepsilon)+ e^{-a\tau(\varepsilon)^{\frac{1}{1-\alpha}}}\tau(\varepsilon)\left|\log(t)-\left(\frac{1}{1-\alpha}\right)\log(\tau(\varepsilon))\right|\leq $$
$$\leq \frac{2}{1-\alpha}(|\log(t)|+|\log(\tau(\varepsilon))|)\tau(\varepsilon),$$
as we wanted to prove.
For the last one, we write,
$$\int_0^\infty e^{-as}l^\alpha_\varepsilon(s)ds=\int_0^{{\tau(\varepsilon)}^{\frac{1}{1-\alpha}}}s^{-\alpha}ds+\int_{{\tau(\varepsilon)}^{\frac{1}{1-\alpha}}}^1s^{-1}\tau(\varepsilon)ds+\tau(\varepsilon)\int_1^\infty e^{-as}s^{-1}ds=$$
$$=\frac{\tau(\varepsilon)}{1-\alpha}+\frac{1}{1-\alpha}|\log(\tau(\varepsilon))|\tau(\varepsilon)+\frac{e^{-a}}{a}\tau(\varepsilon)\leq \frac{2e^{-a}}{a(1-\alpha)}|\log(\tau(\varepsilon))|\tau(\varepsilon).$$
Note that, if $a\geq 1$ then,
$$\int_0^\infty e^{-as}l^\alpha_\varepsilon(s)ds\leq\frac{2}{1-\alpha}|\log(\tau(\varepsilon))|\tau(\varepsilon).$$
This concludes the proof of the result. \end{proof}

\begin{re}\label{simple-estimate}
If  $t=1$, the first  estimate is simplified to
\begin{equation}\label{first-simple-estimate}
\int_0^t(t-s)^{-\gamma}l^\alpha_\eps(s)ds\leq \frac{2^\gamma}{(1-\gamma)(1-\alpha)}|\log(\tau(\eps))| \tau(\eps).
\end{equation}
\end{re}

\section{Existence of Inertial Manifolds}
\label{existence}

Our objective in this section is to construct inertial manifolds $\mathcal{M}_\varepsilon$, for each $0\leq\varepsilon\leq\varepsilon_0$, which will be invariant manifolds for the semi flow generated by  (\ref{problemalimite}) and (\ref{problemaperturbado}), therefore proving Proposition \ref{existenciavariedadinercial}. For this purpose, we will use the Lyapunov-Perron method, see \cite{Sell&You}. This method consists in constructing the inertial manifold as the graph of a Lipschitz map, which is obtained as the fixed point of an appropriate transformation. For that, observe that Lemma \ref{spectrum-behavior} and Remark \ref{autovalores} give us that if the operator $A_0$ has spectral gap, then the operator $A_\varepsilon$ will also have it for $\varepsilon$ small enough. This spectral gap is essential in the construction of the inertial manifold.

To obtain these inertial manifolds $\mathcal{M}_\varepsilon$, $0\leq\varepsilon\leq\varepsilon_0$, consider $m\in\mathbb{N}$ such that $\lambda_m^0<\lambda_{m+1}^0$ (and therefore $\lambda_m^\eps<\lambda_{m+1}^\eps$ for $\eps$ small enough) and denote by $\mathbf{P}_{\mathbf{m}}^{\bm\varepsilon}$ the canonical orthogonal projection onto the eigenfunctions, $\{\varphi^\varepsilon_i\}_{i=1}^m$, corresponding to the first $m$ eigenvalues of the operator $A_\varepsilon $, $0\leq\varepsilon\leq\varepsilon_0$ and $\mathbf{Q}^{\bm\varepsilon}_{\mathbf{m}}$ its orthogonal complement, see (\ref{projectionP}) and (\ref{projectionP0}). The Lyapunov-Perron method obtains $\mathcal{M}_\eps$  as the graph of a function $\Psi_\eps: \mathbf{P_m^\eps} X_\eps^\alpha\to \mathbf{Q_m^\eps} X_\eps^\alpha$ which is obtained as a fixed point of the functional
 \begin{equation}\label{definition-Teps}
 (\mathbf{T}_{\bm \varepsilon}\Psi_\varepsilon)(p^0)=\int_{-\infty}^0e^{A_\varepsilon\mathbf{Q}^{\bm\varepsilon}_{\mathbf{m}} s}\mathbf{Q^\eps_m} F_\eps({p}(s)+\Psi_\varepsilon({p}(s)))ds,
 \end{equation}
where $p(s)\in [\varphi_1^\eps,\ldots,\varphi_m^\eps]$ is the globally defined solution of 
\begin{equation}\label{equationp}
\left\{
\begin{array}{l}
p_t=-A_\varepsilon p+\mathbf{P_m^\eps} F_\eps(p+\Psi_\epsilon(p(t)))\\
p(0)=p^0.
\end{array}
\right.
\end{equation}

Following \cite{Sell&You} it can be seen that:

\begin{prop}\label{existence-inertial-manifold}
Assume hypotheses {\bf (H1)} and {\bf (H2)} are satisfied.  If $m$ is such that 
$$\lambda_{m+1}^0-\lambda_m^0\geq 12L_F[(\lambda_{m+1}^0)^\alpha+(\lambda_{m}^0)^\alpha]$$
$$(\lambda_m^0)^{1-\alpha}\geq 24L_F(1-\alpha)^{-1}$$
then equation \eqref{problemaperturbado} has an inertial manifold $\mathcal{M}_\eps$ given as the graph of a Lipschitz function $\Psi_\eps:[\varphi_1^\eps,\ldots,\varphi_m^\eps]\to \mathbf{Q_m^\eps} X_\eps$ satisfying  
$$\hbox{supp}(\Psi_\eps)\subset \{ \phi\in \mathbf{P_m^\eps} X_\eps^\alpha, \|\phi\|_{X_\eps^\alpha}\leq R\}$$
$$\|\Psi_\eps (p)\|_{X_\eps^\alpha}\leq L_0$$
$$\|\Psi_\eps(p)-\Psi_\eps(p')\|_{X_\eps^\alpha}\leq L_1\|p-p'\|_{X_\eps^\alpha}$$
for certain $L_0$, $L_1$ independent of $\eps$. 

\end{prop}

\begin{proof} Observe that if $m$ is such that the gap conditions of the proposition hold, then for $\eps$ small enough, see Remark \ref{continuity}, we have 
\begin{equation}\label{gaps-eps}
\begin{array}{l} 
(\lambda_m^\eps)^{1-\alpha}\geq 12L_F(1-\alpha)^{-1} \\
\lambda_{m+1}^\eps-\lambda_m^\eps\geq 6L_F[(\lambda_{m+1}^\eps)^\alpha+(\lambda_{m}^\eps)^\alpha]
\end{array}
\end{equation}
which are the gap conditions needed in \cite{Sell&You} to obtain the inertial manifolds for each $\eps$ small enough. \end{proof}

\par\bigskip  With the definition of the isomorphism $j_\eps$, \eqref{definition-jeps}, we may define now the inertial manifolds $\Phi_\eps:\mathbb{R}^m\to \mathbf{Q_m^\eps} X_\eps^\alpha$ as  $\Phi_\eps=\Psi_\eps\circ j_\eps^{-1}$.  Notice also that since $\Psi_\eps$ is a fixed point of $\mathbf{T}_{\bm \eps}$, then the function $\Phi_\eps$ satisfies,

 \begin{equation}\label{definition-Teps}
 (\mathbf{T}_{\bm \varepsilon}\Phi_\varepsilon)(\bar p^0)=\int_{-\infty}^0e^{A_\varepsilon\mathbf{Q}^{\bm\varepsilon}_{\mathbf{m}} s}\mathbf{Q^{\bm \eps}_m} F_\eps\big({p}(s)+\Phi_\varepsilon(j_\eps({p}(s)))\big)ds,
 \end{equation}
where $p(s)$ is the solution of \eqref{equationp} with $p^0=j_\eps^{-1}(\bar p^0)$ or equivalently, $p(s)$ is the solution of
\begin{equation}\label{equationp-modified}
\left\{
\begin{array}{l}
p_t=-A_\varepsilon p+\mathbf{P_m^{\bm \eps}} F_\eps(p+\Phi_\epsilon\circ j_\eps (p(t)))\\
p(0)=j_\eps^{-1}(\bar p^0).
\end{array}
\right.
\end{equation}

 It is an easy exercise now to show that these functions $\Phi_\eps$ are the inertial manifolds from Proposition \ref{existenciavariedadinercial}.

\section{Rate of convergence of the inertial manifolds}
\label{distance}
Once we have proved the existence of the inertial manifolds $\mathcal{M}_{\varepsilon}$, $\varepsilon\geq 0$ and therefore we have fixed the value of $m$,  we are interested in obtaining the rate of convergence of these inertial manifolds as $\varepsilon\rightarrow 0$. To accomplish this, we will need to subtract the integral expressions \eqref{definition-Teps} for $\eps=0$ and $\eps>0$ and make several estimates on these differences. Therefore, we will need first to obtain good estimates on the behavior of the semigroup acting in the spaces $\mathbf{P_m^{\bm\eps}}X_\eps^\alpha$ and  $\mathbf{Q_m^{\bm \eps}}X_\eps^\alpha$.  

Since the value of $m$ is fixed and we have the gap condition from Proposition \ref{existence-inertial-manifold} without loss of generality we will assume that $\lambda_{m+1}^\eps-\lambda_m^\eps\geq 3$ for all $0\leq \eps\leq \eps_0$. This allows us to construct the  following rectangular curve, encircling the first $m$ eigenvalues: 
$$\Gamma=\Gamma^1\cup\Gamma^2\cup\Gamma^3\cup\Gamma^4,$$
where,
$$\Gamma^1=\{\lambda\in\mathbb{C}: Re(\lambda)=-\lambda_1^0+1\,\,\textrm{and}\,\,  |Im(\lambda)|\leq 1\},$$
$$\Gamma^2=\{\lambda\in\mathbb{C}: -\lambda_m^0-1\leq Re(\lambda)\leq-\lambda_1^0+1\,\,\textrm{and}\,\, Im(\lambda)=1\},$$
$$\Gamma^3=\{\lambda\in\mathbb{C}: Re(\lambda)=-\lambda_m^0-1\,\,\textrm{and}\,\, |Im(\lambda)|\leq 1\},$$
$$\Gamma^4=\{\lambda\in\mathbb{C}: -\lambda_m^0-1\leq Re(\lambda)\leq-\lambda_1^0+1\,\,\textrm{and}\,\, Im(\lambda)=-1\}.$$

We can prove now,
\begin{lem}\label{estimacionsemigruposlinealesproyectadosP}
Let hypothesis {\bf (H1)} be satisfied and let $\Gamma$ the curve defined above. Then,
$$\|e^{-A_\varepsilon t}\mathbf{P}_{\mathbf{m}}^{\bm\varepsilon}E-E e^{-A_0 t}\mathbf{P}_{\mathbf{m}}^{\mathbf{0}}\|_{\mathcal{L}(X_0, X^\alpha_\varepsilon)}\leq C_4 e^{-(\lambda_m^0+1)t}\tau(\varepsilon),\qquad \,\,t\leq 0,$$
with $C_4=\frac{|\Gamma|}{2\pi}\sup_{\lambda\in\Gamma} C^\varepsilon_3(\lambda).$
\end{lem}
\begin{proof}

Since the curve $\Gamma$ contains the first $m$ eigenvalues of $-A_\eps$, $0\leq \eps\leq \eps_0$, then
$$e^{-A_\varepsilon t}\mathbf{P}_{\mathbf{m}}^{\bm\varepsilon}E-E e^{-A_0 t}\mathbf{P}_{\mathbf{m}}^{\mathbf{0}}=\frac{1}{2\pi i}\int_{\Gamma}\left((\lambda I+A_\varepsilon)^{-1}E-E(\lambda I+A_0)^{-1}\right)e^{\lambda t}d\lambda.$$
So,
$$\|e^{-A_\varepsilon t}\mathbf{P}_{\mathbf{m}}^{\bm\varepsilon}E-E e^{-A_0 t}\mathbf{P}_{\mathbf{m}}^{\mathbf{0}}\|_{\mathcal{L}(X_0, X_\varepsilon^\alpha)}$$
$$\leq\frac{1}{2\pi}\int_{\Gamma}\|(\lambda I+A_\varepsilon)^{-1}E-E(\lambda I+A_0)^{-1}\|_{\mathcal{L}(X_0, X_\varepsilon^\alpha)}|e^{\lambda t}|d\lambda.$$
Applying Lemma \ref{estimacionlandas}, for $t\leq 0$ we have,
$$\|e^{-A_\varepsilon t}\mathbf{P}_{\mathbf{m}}^{\bm\varepsilon}E-E e^{-A_0 t}\mathbf{P}_{\mathbf{m}}^{\mathbf{0}}\|_{\mathcal{L}(X_0, X_\varepsilon^\alpha)}$$
$$\leq\frac{|\Gamma|}{2\pi}\sup_{\lambda\in\Gamma} C^\varepsilon_3(\lambda)\tau(\varepsilon)\sup_{\lambda\in\Gamma}e^{Re (\lambda) t}=C_4e^{-(\lambda_m^0+1)t}\tau(\varepsilon),$$
with $C_4=\frac{|\Gamma|}{2\pi}\sup_{\lambda\in\Gamma} C^\varepsilon_3(\lambda)$ and $|\Gamma|$ the length of the curve $\Gamma$.
\end{proof}
\vspace{0.4cm}

With respect to the behavior of the linear semigroup in the subspace $\mathbf{Q}^{\bm\varepsilon}_{\mathbf{m}}X_\eps^\alpha$, notice that we have the expression 
$$e^{-A_\varepsilon \mathbf{Q}^{\bm\varepsilon}_{\mathbf{m}} t}u=\sum_{i=m+1}^\infty e^{-\lambda_i^\varepsilon t}(u, \varphi_i^\varepsilon)\varphi_i^\varepsilon.$$
Hence, following a similar proof as Lemma \ref{estimacionsemigrupolineal}, we get
$$\|e^{-A_\varepsilon \mathbf{Q}^{\bm\varepsilon}_{\mathbf{m}} t}\|_{\mathcal{L}(X_\varepsilon, X_\varepsilon)}\leq e^{-\lambda_{m+1}^\varepsilon t},$$ 
and,
\begin{equation}\label{semigrupoproyectado}
\|e^{-A_\varepsilon \mathbf{Q}^{\bm\varepsilon}_{\mathbf{m}} t}\|_{\mathcal{L}(X_\varepsilon, X_\varepsilon^\alpha)}\leq e^{-\lambda_{m+1}^\varepsilon t}\left(\max\{\lambda_{m+1}^\varepsilon, \frac{\alpha}{t}\}\right)^\alpha,
\end{equation}
for $t\geq 0.$
\vspace{0.5cm}

 Before continuing, we now present technical lemmas henceforward needed.  
 \begin{lem}\label{integralmaximo}
 Let $a$ be a positive constant, $a>0$, $\alpha\in(0, 1)$ and $\lambda >0$ a positive real number. We have the following estimate,
 $$\int_0^\infty e^{-as} \left(\max\{\lambda, \frac{\alpha}{s}\}\right)^\alpha ds\leq(1-\alpha)^{-1}\lambda^{\alpha-1}+\lambda^\alpha a^{-1}.$$
 \end{lem}
  \begin{proof}
  Let $\alpha\in(0,1)$ and $\lambda$ a  real positive number. Then we know that
  $$\max\{\lambda, \frac{\alpha}{s}\}=
  \begin{cases}
  \frac{\alpha}{s} &\mbox{if \hspace{0.3cm}} 0<s\leq\frac{\alpha} {\lambda} \\
  \lambda&\mbox{if \hspace{0.3cm}} \frac{\alpha}{\lambda}< s<\infty.
  \end{cases}
  $$
  So,
  $$\int_0^\infty\left(\max\{\lambda, \frac{\alpha}{s}\}\right)^\alpha e^{-as}ds=\int_0^{\frac{\alpha}{\lambda}} \left(\frac{\alpha}{s}\right)^\alpha e^{-as}ds+ \int_{\frac{\alpha}{\lambda}}^\infty \lambda^\alpha e^{-as}ds=$$
  $$=\alpha^\alpha\int_0^{\frac{\alpha}{\lambda}} s^{-\alpha} e^{-as}ds+ \lambda^\alpha\int_{\frac{\alpha}{\lambda}}^\infty  e^{-as}ds=$$
  $$=\alpha^\alpha\left(\frac{\alpha}{\lambda}\right)^{1-\alpha}(1-\alpha)^{-1}+\lambda^\alpha e^{-\frac{a\alpha}{\lambda}}a^{-1}\leq $$
  $$\leq(1-\alpha)^{-1}\lambda^{\alpha-1}+ \lambda^\alpha a^{-1},$$
  as we wanted to prove.\end{proof}
  
  \par\bigskip
  
  Now we want to compare both semigroups $e^{-A_\eps t}$ and $e^{-A_0 t}$ in $\mathbf{Q}_{\mathbf{m}}^{\bm\varepsilon}X_\eps^\alpha$ and $\mathbf{Q}_{\mathbf{m}}^{\bm 0}X_0^\alpha$. For this, we define first the curve
  $\Gamma_m$ which is given by the boundary of $\Sigma_{b, \phi}=\{\lambda\in\mathbb{C}: |arg(\lambda-b)|\leq\pi-\phi\}$, with $\phi=\frac{\pi}{4}$ and $b=-\lambda_{m+1}^0+1$. That is,
$$\Gamma_m=\Gamma_m^1\cup\Gamma_m^2=\{b+re^{-i(\pi-\phi)}: 0\leq r<\infty\}\cup\{b+re^{i(\pi-\phi)}: 0\leq r<+\infty\},$$
oriented such that the imaginary part grows as $\lambda$ runs in $\Gamma$.

We have the following estimates, 
\begin{lem}\label{estimacionsemigruposlinealesproyectados}
Let hypothesis {\bf (H1)} be satisfied. If, for $t>0$, as before we denote by
$$ l_\varepsilon^\alpha(t):=\min\{t^{-1}\tau(\varepsilon), t^{-\alpha}\},$$ 
then, for each   $t>0$,
$$\|e^{-A_\varepsilon t}\mathbf{Q}_{\mathbf{m}}^{\bm\varepsilon}E-E e^{-A_0 t}\mathbf{Q}_{\mathbf{m}}^{\mathbf{0}}\|_{\mathcal{L}(X_0, X^\alpha_\varepsilon)}\leq C_5 e^{-(\lambda_{m+1}^0-1) t}l_\varepsilon^\alpha(t),$$
where $C_5=\max\{\frac{\sup_{\lambda\in\Gamma_m} C_3^\eps(\lambda)}{\pi cos(\phi) }, 4\}$ and $C_3^\eps(\lambda)$ is defined in Lemma \ref{estimacionlandas}.
\end{lem}  
\begin{proof}
From Lemma \ref{spectrum-behavior} and Remark \ref{autovalores}, we know that there is a real number $\varepsilon_0=\varepsilon_0(m)$ such that, for $0\leq\varepsilon\leq\varepsilon_0$, there is a gap between the  $mth$-eigenvalue, $-\lambda_m^\varepsilon$, and $m+1$-eigenvalues, $-\lambda_{m+1}^\varepsilon$, of $-A_\varepsilon$. We denote by $\Gamma_m$ the boundary of $\Sigma_{b, \phi}=\{\lambda\in\mathbb{C}: |arg(\lambda-b)|\leq\pi-\phi\}$, with $\phi=\frac{\pi}{4}$ and $b=-\lambda_{m+1}^0+1$. That is,
$$\Gamma_m=\Gamma_m^1\cup\Gamma_m^2=\{b+re^{-i(\pi-\phi)}: 0\leq r<\infty\}\cup\{b+re^{i(\pi-\phi)}: 0\leq r<+\infty\},$$
oriented such that the imaginary part grows as $\lambda$ runs in $\Gamma$. 

With this, 

$$e^{-A_\varepsilon t}\mathbf{Q}_{\mathbf{m}}^{\bm\varepsilon}E-E e^{-A_0 t}\mathbf{Q}_{\mathbf{m}}^{\mathbf{0}}= \frac{1}{2\pi i}\int_{\Gamma_m}\left((\lambda+A_\varepsilon)^{-1}E-E(\lambda+A_0)^{-1}\right)e^{\lambda t}d\lambda.$$
Then,
$$\|e^{-A_\varepsilon t}\mathbf{Q}_{\mathbf{m}}^{\bm\varepsilon}E-E e^{-A_0 t}\mathbf{Q}_{\mathbf{m}}^{\mathbf{0}}\|_{\mathcal{L}(X_0, X_\varepsilon^\alpha)}$$
$$\leq \frac{1}{2\pi}\left|\int_{\Gamma_m}\|\left((\lambda+A_\varepsilon)^{-1}E-E(\lambda+A_0)^{-1}\right)\|_{\mathcal{L}(X_0, X_\varepsilon^\alpha)}|e^{\lambda t}|d\lambda\right|,$$
applying Lemma \ref{estimacionlandas}
$$\|e^{-A_\varepsilon t}\mathbf{Q}_{\mathbf{m}}^{\bm\varepsilon}E-E e^{-A_0 t}\mathbf{Q}_{\mathbf{m}}^{\mathbf{0}}\|_{\mathcal{L}(X_0, X_\varepsilon^\alpha)}$$
$$\leq \frac{\sup_{\lambda\in\Gamma_m} C_3^\eps(\lambda)\tau(\varepsilon)}{2\pi}\left|\int_{\Gamma_m}|e^{\lambda t}| d\lambda \right|= \frac{\sup_{\lambda\in\Gamma_m} C_3^\eps(\lambda)\tau(\varepsilon)}{\pi}\left|\int_{\Gamma_m^2}|e^{\lambda t}| d\lambda \right|.$$
Since $\lambda\in\Gamma_m^2$,
$$|e^{\lambda t}|=e^{(b-rcos(\phi))t }.$$
So,
$$\|e^{-A_\varepsilon t}\mathbf{Q}_{\mathbf{m}}^{\bm\varepsilon}E-E e^{-A_0 t}\mathbf{Q}_{\mathbf{m}}^{\mathbf{0}}\|_{\mathcal{L}(X_0, X_\varepsilon^\alpha)}$$
$$\leq \frac{\sup_{\lambda\in\Gamma_m} C_3^\eps(\lambda)\tau(\varepsilon)}{\pi}\int_0^\infty e^{(b-rcos(\phi))t}|e^{-i(\pi-\phi)}|dr.$$
We make the change of variables $(-b+rcos(\phi))t=z$,
$$\|e^{-A_\varepsilon t}\mathbf{Q}_{\mathbf{m}}^{\bm\varepsilon}E-E e^{-A_0 t}\mathbf{Q}_{\mathbf{m}}^{\mathbf{0}}\|_{\mathcal{L}(X_0, X_\varepsilon^\alpha)}\leq \frac{\sup_{\lambda\in\Gamma_m} C_3^\eps(\lambda)\tau(\varepsilon)}{\pi cos(\phi) t}\int_{-bt}^\infty e^{-z}dz=$$
$$= \frac{\sup_{\lambda\in\Gamma_m} C_3^\eps(\lambda)}{\pi cos(\phi) } t^{-1}e^{-(\lambda_{m+1}^0-1)t}\tau(\varepsilon). $$

On the other side, we know that,
$$\|e^{-A_\varepsilon t}\mathbf{Q}_{\mathbf{m}}^{\bm\varepsilon}E-E e^{-A_0 t}\mathbf{Q}_{\mathbf{m}}^{\mathbf{0}}\|_{\mathcal{L}(X_0, X_\varepsilon^\alpha)}\leq $$
$$\|e^{-A_\varepsilon t}\mathbf{Q}_{\mathbf{m}}^{\bm\varepsilon}E\|_{\mathcal{L}(X_0, X^\alpha_\varepsilon)}+\|E e^{-A_0 t}\mathbf{Q}_{\mathbf{m}}^{\mathbf{0}}\|_{\mathcal{L}(X_0, X^\alpha_\varepsilon)}.$$
Then, by (\ref{semigrupoproyectado}),
$$\leq 2e^{-\lambda_{m+1}^\varepsilon t}\left(\max\{\lambda_{m+1}^\varepsilon, \frac{\alpha}{t}\}\right)^\alpha+2 e^{-\lambda_{m+1}^0 t}\left(\max\{\lambda_{m+1}^0, \frac{\alpha}{t}\}\right)^\alpha\leq$$
$$\leq 4e^{-(\lambda_{m+1}^0-1) t}\left(\max\{(\lambda_{m+1}^0+1)^\alpha,t^{-\alpha}\}\right).$$

So, if we put everything together,
$$\|e^{-A_\varepsilon t}\mathbf{Q}_{\mathbf{m}}^{\bm\varepsilon}E-E e^{-A_0 t}\mathbf{Q}_{\mathbf{m}}^{\mathbf{0}}\|_{\mathcal{L}(X_0, X^\alpha_\varepsilon)}$$
$$\leq C_5 \min\left\{t^{-1}\tau(\varepsilon), \max\{(\lambda_{m+1}^0+1)^\alpha, t^{-\alpha}\} \right\}e^{-(\lambda_{m+1}^0 -1) t}=$$
$$=C_5 \min\left\{t^{-1}\tau(\varepsilon), t^{-\alpha}\right\}e^{-(\lambda_{m+1}^0 -1) t}= C_5 l_\varepsilon^\alpha e^{-(\lambda_{m+1}^0 -1) t},$$
as we wanted to prove.
\end{proof}

We may show now the following result.

\begin{lem}\label{normacoordenadas}
Let $w_\varepsilon\in\mathbf{P}_{\mathbf{m}}^{\bm\varepsilon} X_\varepsilon$ and $w_0\in\mathbf{P}_{\mathbf{m}}^{\mathbf{0}}X_0$. Then, for $\eps$ small enough and for $0\leq\alpha<1$,
$$|j_\varepsilon(w_\varepsilon)-j_0(w_0)|_\alpha\leq 3\|w_\varepsilon-Ew_0\|_{X_\varepsilon^\alpha}+ 3C_P\tau(\varepsilon)\|w_0\|_{X_0}.$$
\end{lem}
  \begin{proof}
Since $\varphi_i^0=\mathbf{P_m^0}(\varphi_i^0)$, then if we denote by $j_\eps(w_\eps)=\bar p_\eps$ and $j_0(w_0)=\bar p_0$ 
$$w_\varepsilon-Ew_0=\sum_{I=1}^mp_i^\varepsilon \mathbf{P}_{\mathbf{m}}^{\bm\varepsilon}(E\varphi_i^0)-E\sum_{i=1}^mp_i^0\mathbf{P}_{\mathbf{m}}^{\mathbf{0}}\varphi_i^0$$
$$=(\mathbf{P}_{\mathbf{m}}^{\bm\varepsilon}E- E \mathbf{P_m^0})\left(\sum_{I=1}^mp_i^\varepsilon \varphi_i^0 \right)+E\sum_{i=1}^m(p_i^\eps-p_i^0)\varphi_i^0$$
Applying the operator $M$ and using that $M\circ E=I$, we get
$$\sum_{i=1}^m(p_i^\eps-p_i^0)\varphi_i^0=M(w_\varepsilon-Ew_0)- M(\mathbf{P}_{\mathbf{m}}^{\bm\varepsilon}E- E \mathbf{P_m^0})\left(\sum_{I=1}^mp_i^\varepsilon \varphi_i^0 \right)$$
Taking the $X_0^\alpha$ norm in the last expression and with \eqref{cotaextensionproyeccion},  Lemma \ref{estimacionProyecciones} and \eqref{normajepsilon}, we get
$$|\bar p_\eps-\bar p_0|_\alpha\leq 2\|w_\varepsilon-Ew_0\|_{X_\eps^\alpha}+2C_p\tau(\eps)|\bar p_\eps|$$
$$\leq 
2\|w_\varepsilon-Ew_0\|_{X_\eps^\alpha}+2C_p\tau(\eps)|\bar p_\eps-\bar p_0|+2C_p\tau(\eps)|\bar p_0|.$$
From here, we get 
$$|\bar p_\eps-\bar p_0|_\alpha\leq \frac{2}{1-2C_P\tau(\eps)}\|w_\varepsilon-Ew_0\|_{X_\eps^\alpha}+\frac{2}{1-2C_P\tau(\eps)}C_p\tau(\eps)|\bar p_0|.$$
Taking $\eps$ small enough so that $\frac{2}{1-2C_P\tau(\eps)}\leq 3$ and since $|\bar p_0|=\|w_0\|_{X_0}$, we prove the result. \end{proof}
\par\bigskip

Next, we introduce some technical results.
\begin{lem}\label{normap}
For every $\Phi_\varepsilon\in\mathcal{F}_{L}$  with $L\leq 1$ and any $\bar{p}^0\in\mathbb{R}^m$,  
if $p_\eps(t)$ is the solution of \eqref{equationp-modified}, we have,
$$\|p_\eps(t)\|_{X^\alpha_\eps}\leq \left(\|j_\eps^{-1}(\bar p^0)\|_{X_\eps^\alpha}+\frac{C_F}{(\lambda_m^\varepsilon)^{1-\alpha}}\right)e^{-\lambda_m^\varepsilon t},\qquad t\leq 0,$$
\end{lem}
  \begin{proof}
By the variation of constant formula for $t\leq0$,
$$\|p_\eps(t)\|_{X_\eps^\alpha}\leq \|e^{-A_\eps t}j_\eps^{-1}(\bar{p}^0)\|_{X_\eps^\alpha}+\int_t^0\|e^{-A_\varepsilon  (t-s)}\mathbf{P_m^{\bm \eps}}F_\varepsilon({p}_\varepsilon(s)+\Phi_\varepsilon j_\eps ({p}_\varepsilon(s)))\|_{X_\eps^\alpha}ds$$
$$\leq e^{-\lambda_m^\eps t}\|j_\eps^{-1}(\bar p^0)\|_{X_\eps^\alpha} +\int_t^0 e^{-\lambda_m^\eps(t-s)}(\lambda_m^\eps)^\alpha\|\mathbf{P_m^{\bm \eps}}F_\varepsilon({p}_\varepsilon(s)+\Phi_\varepsilon j_\eps ({p}_\varepsilon(s)))\|_{X_\eps}ds$$
$$\leq e^{-\lambda_m^\eps t}\|j_\eps^{-1}(\bar p^0)\|_{X_\eps^\alpha} +\int_t^0 e^{-\lambda_m^\eps(t-s)}(\lambda_m^\eps)^\alpha C_Fds$$
$$\leq \left(\|j_\eps^{-1}(\bar p^0)\|_{X_\eps^\alpha}+\frac{C_F}{(\lambda_m^\varepsilon)^{1-\alpha}}\right)e^{-\lambda_m^\varepsilon t}$$
\end{proof}

Let $\Phi_0$ and $\Phi_\varepsilon$ be the inertial manifolds constructed above.  If $\bar{p}^0\in\mathbb{R}^m$, we denote by $p_0(t)\in \mathbf{P_m^0}X_0^\alpha$ and $p_\eps(t)\in \mathbf{P_m^\eps}(X_\eps^\alpha)$ the solutions of the initial value problems, respectively,
\begin{equation}\label{problemainicialp-0}
{p_0}_t=-A_0{p_0}+\mathbf{P_m^0}F_0(p_0+\Phi_0(j_0(p_0))),\qquad {p_0}(0)=j_0^{-1}\bar{p}^0,
\end{equation}
and
\begin{equation}\label{problemainicialp}
{p_\eps}_t=-A_\eps{p_\eps}+\mathbf{P_m^{\bm \eps}}F_\eps(p_\eps+\Phi_\eps(j_\eps(p_\eps))),\qquad {p_\eps}(0)=j_\eps^{-1}\bar{p}^0\end{equation}

We have now, 
\begin{lem}\label{distanciapes}
With the notations above, we have, for $t\leq 0$,
$$\|p_\eps(t)-Ep_0(t)\|_{X_\eps^\alpha}\leq  \left(\frac{1}{12}\sup_{\bar p\in \mathbb{R}^m}\|\Phi_\eps(\bar p)-E \Phi_0(\bar p)\|_{X_\eps^\alpha}+\right.$$
$$+\rho(\eps)+K_2(|t|+e^{-2 t})\tau(\eps)\Big)e^{-(\lambda_m^\eps+4L_F(\lambda_m^\eps)^\alpha) t}$$
with $K_2=(6(\lambda_m^0)^\alpha L_FC_P+C_4)(|\bar p^0|+C_F)$ and $C_4$ is the constant from Lemma  \ref{estimacionsemigruposlinealesproyectadosP}.

\end{lem}
  \begin{proof}
To simplify the notation below, we denote by $\tilde F_\eps=F_\eps(p_\eps(s)+\Phi_\eps(j_\eps(p_\eps(s))))$ and
similarly, $\tilde F_0=F_0(p_0(s)+\Phi_0(j_0(p_0(s))))$. 
By the variation of constants formula applied to (\ref{problemainicialp-0}) and (\ref{problemainicialp}) we get

$$p_\eps(t)-Ep_0(t)=e^{-A_\eps t}j_\eps^{-1}(\bar p^0)-Ee^{-A_0 t}j_0^{-1}(\bar p^0)$$
$$+ \int_0^t
\left(
e^{-A_\eps (t-s)}\mathbf{P_m^{\bm \eps}}\tilde F_\eps-
Ee^{-A_0 (t-s)}\mathbf{P_m^0}\tilde F_0\right)ds$$

$$=e^{-A_\eps t}j_\eps^{-1}(\bar p^0)-Ee^{-A_0 t}j_0^{-1}(\bar p^0)+
\int_0^t
e^{-A_\eps (t-s)}\mathbf{P_m^{\bm \eps}}(\tilde F_\eps-E\tilde F_0)ds$$
$$+\int_0^t (e^{-A_\eps (t-s)}\mathbf{P_m^{\bm \eps}} E- Ee^{-A_0 (t-s)}\mathbf{P_m^0})\tilde F_0 ds=I_1+I_2+I_3$$

Observe that, with the definition of $j_\eps$ and with the aid of Lemma \ref{estimacionsemigruposlinealesproyectadosP}, we get 
$$\|I_1\|_{X_\eps^\alpha}=\|(e^{-A_\eps t}\mathbf{P_m^\eps}E-Ee^{-A_0 t}\mathbf{P_m^0})(\sum_{i=1}^m p_i^0\varphi_i^0)\|_{X_\eps^\alpha}\leq C_4e^{-(\lambda_m^0+1)t}\tau(\eps)|\bar p^0|$$

Moreover, we have
\begin{equation}\label{F-decomposition}
\begin{array}{l}
\tilde F_\eps-E \tilde F_0=F_\eps(p_\eps +\Phi_\eps(j_\eps(p_\eps)))-F_\eps(Ep_0 +\Phi_\eps(j_\eps(p_\eps)))\\ \\
\qquad\qquad+F_\eps(Ep_0+\Phi_\eps(j_\eps(p_\eps)))- F_\eps(Ep_0+\Phi_\eps(j_0(p_0)))\\ \\
\qquad\qquad+F_\eps(Ep_0+\Phi_\eps(j_0(p_0))- F_\eps(Ep_0+E\Phi_0(j_0(p_0)))\\ \\
\qquad\qquad+F_\eps(Ep_0+E\Phi_0(j_0(p_0))- EF_0(p_0+\Phi_0(j_0(p_0)))
\end{array}
\end{equation}
which implies
$$\|\tilde F_\eps-E \tilde F_0\|_{X_\eps}\leq L_F\|p_\eps-Ep_0\|_{X_\eps^\alpha}+L_F\cdot L|j_\eps(p_\eps)-j_0(p_0))|_{\alpha}$$
$$+L_F\sup_{\bar p\in \mathbb{R}^m}\|\Phi_\eps(\bar p)-E \Phi_0(\bar p)\|_{X_\eps^\alpha}+\rho(\eps)
$$

Taking into account Lemma \ref{normacoordenadas} , we get
$$\|\tilde F_\eps-E \tilde F_0\|_{X_\eps}\leq 4L_F\|p_\eps-Ep_0\|_{X_\eps^\alpha}+3L_FC_P\tau(\eps)\|p_0\|_{X_0}$$
$$+L_F\sup_{\bar p\in \mathbb{R}^m}\|\Phi_\eps(\bar p)-E \Phi_0(\bar p)\|_{X_\eps^\alpha}+\rho(\eps) $$
which implies with Lemma \ref{normap} and using that $\lambda_m^\eps\geq 1$, 
\begin{equation}\label{estimate}
\begin{array}{l}
\|\tilde F_\eps-E \tilde F_0\|_{X_\eps}\leq  4L_F\|p_\eps-Ep_0\|_{X_\eps^\alpha}+3L_FC_P\tau(\eps)(|\bar p^0|+C_F)e^{-\lambda_m^\eps s}+\\ \\\qquad\qquad +\displaystyle L_F\sup_{\bar p\in \mathbb{R}^m}\|\Phi_\eps(\bar p)-E \Phi_0(\bar p)\|_{X_\eps^\alpha}+\rho(\eps)
\end{array}
\end{equation}
In particular, we obtain:

$$\|I_2\|_{X_\eps^\alpha}\leq (\lambda_m^\eps)^\alpha \int_t^0e^{-\lambda_m^\eps(t-s)} \|\tilde F_\eps-E \tilde F_0\|_{X_\eps}ds$$
That is,

$$
\begin{array}{l}
\displaystyle \|I_2\|_{X_\eps^\alpha}\leq 4L_F(\lambda_m^\eps)^\alpha \int_t^0e^{-\lambda_m^\eps(t-s)} \|p_\eps(s)-Ep_0(s)\|_{X_\eps^\alpha}ds\\
\displaystyle\qquad\qquad+(\lambda_m^\eps)^\alpha3L_FC_P(|\bar p^0|+C_F) |t|\tau(\eps)e^{-\lambda_m^\eps t}+\\
\displaystyle\qquad\qquad+(\lambda_m^\eps)^\alpha \left(L_F \sup_{\bar p\in \mathbb{R}^m}\|\Phi_\eps(\bar p)-E \Phi_0(\bar p)\|_{X_\eps^\alpha}+\rho(\eps)\right)\frac{e^{-\lambda_m^\eps t}-1}{\lambda_m^\eps} \\
\qquad\qquad \leq \displaystyle \left(\frac{L_F}{(\lambda_m^\eps)^{1-\alpha}} \sup_{\bar p\in \mathbb{R}^m}\|\Phi_\eps(\bar p)-E \Phi_0(\bar p)\|_{X_\eps^\alpha}+\rho(\eps)+K_1|t|\tau(\eps)\right) e^{-\lambda_m^\eps t}+\\ 
\qquad\qquad \displaystyle+4L_F(\lambda_m^\eps)^\alpha \int_t^0e^{-\lambda_m^\eps(t-s)} \|p_\eps(s)-Ep_0(s)\|_{X_\eps^\alpha}ds \\
\qquad\qquad \leq \displaystyle \left(\frac{1}{12} \sup_{\bar p\in \mathbb{R}^m}\|\Phi_\eps(\bar p)-E \Phi_0(\bar p)\|_{X_\eps^\alpha}+\rho(\eps)+K_1|t|\tau(\eps)\right) e^{-\lambda_m^\eps t}+\\ 
\qquad\qquad \displaystyle+4L_F(\lambda_m^\eps)^\alpha \int_t^0e^{-\lambda_m^\eps(t-s)} \|p_\eps(s)-Ep_0(s)\|_{X_\eps^\alpha}ds

\end{array}
$$
where we have denoted by $K_1=6(\lambda_m^0)^\alpha L_FC_P(|\bar p^0|+C_F)$ and we have used that $\lambda_m^\eps>1$
and $(\lambda_m^\eps)^\alpha\leq 2 (\lambda_m^0)^\alpha$

Finally, 
$$\|I_3\|_{X_\eps^\alpha}\leq C_4\tau(\eps) C_F\int_t^0 e^{-(\lambda_m^0+1)(t-s)}ds\leq C_4\tau(\eps) C_Fe^{-(\lambda_m^0+1)t}$$

Putting the three expressions together, we get
$$
\begin{array}{l}
\displaystyle 
\|p_\eps(t)-Ep_0(t)\|_{X_\eps^\alpha}\leq C_4(|\bar p^0|+C_F)e^{-(\lambda_m^0+1)t}\tau(\eps)+ \\
\displaystyle \left(\frac{1}{12} \sup_{\bar p\in \mathbb{R}^m}\|\Phi_\eps(\bar p)-E \Phi_0(\bar p)\|_{X_\eps^\alpha}+\rho(\eps)+K_1|t|\tau(\eps)\right) e^{-\lambda_m^\eps t}\\ 
\displaystyle+4L_F(\lambda_m^\eps)^\alpha \int_t^0e^{-\lambda_m^\eps(t-s)} \|p_\eps(s)-Ep_0(s)\|_{X_\eps^\alpha}ds.
\end{array}
$$
Multiplying this inequality by $e^{\lambda_m^\eps t}$, denoting by $h(t)=e^{\lambda_m^\eps t}\|p_\eps(t)-Ep_0(t)\|_{X_\eps^\alpha}$ and assuming $\eps$ is small enough so that $|\lambda_m^\eps-\lambda_m^0|<1$, we may write 
$$
\begin{array}{l}
\displaystyle 
h(t)\leq 
\displaystyle \left(\frac{1}{12} \sup_{\bar p\in \mathbb{R}^m}\|\Phi_\eps(\bar p)-E \Phi_0(\bar p)\|_{X_\eps^\alpha}+\rho(\eps)+K_2(|t|+e^{-2 t})\tau(\eps)\right) \\
\displaystyle \qquad\qquad+4L_F(\lambda_m^\eps)^\alpha \int_t^0h(s) ds
\end{array}
$$
where $K_2=(6(\lambda_m^0)^\alpha L_FC_P+C_4)(|\bar p^0|+C_F)$.   Applying Gronwall inequality, we get, 
$$
h(t)\leq  \left(\frac{1}{12}\sup_{\bar p\in \mathbb{R}^m}\|\Phi_\eps(\bar p)-E \Phi_0(\bar p)\|_{X_\eps^\alpha}+\rho(\eps)+K_2(|t|+e^{-2 t})\tau(\eps)\right)e^{-4L_F(\lambda_m^\eps)^\alpha t}$$
which implies that
$$
\|p_\eps(t)-Ep_0(t)\|_{X_\eps^\alpha}\leq  \Big(\frac{1}{12}\sup_{\bar p\in \mathbb{R}^m}\|\Phi_\eps(\bar p)-E \Phi_0(\bar p)\|_{X_\eps^\alpha}+\rho(\eps)$$
$$\qquad\qquad+K_2(|t|+e^{-2t})\tau(\eps)\Big)e^{-(\lambda_m^\eps+4L_F(\lambda_m^\eps)^\alpha) t}$$
which shows the result.\end{proof}

\par\bigskip\bigskip

With these results, we have all the needed tools to estimate the rate of convergence of the inertial manifolds,  proving the main result of the article 

\par\bigskip\bigskip

  \begin{proof}
Notice that we have 
 \begin{equation}\label{phi-eps}
 \Phi_0(\bar p^0)=\int_{-\infty}^0e^{A_0 s}\mathbf{Q}^0_m F_0\big({p}_0(s)+\Phi_0(j_0({p}_0(s)))\big)ds,
 \end{equation}
and
 \begin{equation}\label{phi-eps}
 \Phi_\varepsilon(\bar p^0)=\int_{-\infty}^0e^{A_\varepsilon  s}\mathbf{Q}^\eps_m F_\eps\big({p}_\eps(s)+\Phi_\varepsilon(j_\eps({p_\eps}(s)))\big)ds,
 \end{equation}
where $p_0(s)$ and $p_\eps(s)$ are the solutions of \eqref{problemainicialp-0} and \eqref{problemainicialp}.  Denoting, as in the proof of the previous Lemma, 
$\tilde F_\eps=F_\eps(p_\eps(s)+\Phi_\eps(j_\eps(p_\eps(s))))$ and
 $\tilde F_0=F_0(p_0(s)+\Phi_0(j_0(p_0(s))))$ 
$$\Phi_\eps(\bar p^0)-E\Phi_0(\bar p^0)=\int_{-\infty}^0\left(e^{A_\varepsilon  s}\mathbf{Q}^\eps_m \tilde F_\eps -Ee^{A_0 s}\mathbf{Q}^0_m \tilde F_0\right)ds=$$ 
$$=\int_{-\infty}^0 e^{A_\varepsilon  s}\mathbf{Q}^\eps_m ( \tilde F_\eps -E \tilde F_0)ds+ \int_{-\infty}^0 \left(e^{A_\varepsilon  s}\mathbf{Q}^\eps_m E- Ee^{A_0 s}\mathbf{Q}^0_m\right)\tilde F_0ds=I_1+I_2.$$
With \eqref{semigrupoproyectado} 
$$\|I_1\|_{X_\eps^\alpha}\leq \int_{-\infty}^0 e^{\lambda_{m+1}^\eps s}\left(\max\{\lambda_{m+1}^\varepsilon, \frac{\alpha}{t}\}\right)^\alpha\|\tilde F_\eps -E \tilde F_0\|_{X_\eps}ds.$$
Now, with the decomposition as in \eqref{F-decomposition} and with \eqref{estimate} and denoting by $\|E\Phi_0-\Phi_\varepsilon\|_\infty=\|E\Phi_0-\Phi_\varepsilon\|_{L^\infty(\mathbb{R}^m, X^\alpha_\varepsilon)}$, we obtain
$$
\|I_1\|_{X_\eps^\alpha}\leq\int_{-\infty}^0 e^{\lambda_{m+1}^\eps s}\left(\max\{\lambda_{m+1}^\varepsilon, \frac{\alpha}{s}\}\right)^\alpha \Big[ 4L_F\|p_\eps(s)-Ep_0(s)\|_{X_\eps^\alpha}$$
$$+3L_FC_P\tau(\eps)(|\bar p^0|+C_F)e^{-\lambda_m^\eps s}+L_F\|E\Phi_0-\Phi_\varepsilon\|_\infty+\rho(\eps)\Big]ds
$$

$$=4L_F\int_{-\infty}^0 e^{\lambda_{m+1}^\eps s}\left(\max\{\lambda_{m+1}^\varepsilon, \frac{\alpha}{s}\}\right)^\alpha \|p_\eps(s)-Ep_0(s)\|_{X_\eps^\alpha}ds+$$
$$\qquad + 3L_FC_P\tau(\eps)(|\bar p^0|+C_F)\int_{-\infty}^0 e^{(\lambda_{m+1}^\eps-\lambda_m^\eps) s}\left(\max\{\lambda_{m+1}^\varepsilon, \frac{\alpha}{s}\}\right)^\alpha ds+$$
$$+\rho(\eps)\int_{-\infty}^0 e^{\lambda_{m+1}^\eps s}\left(\max\{\lambda_{m+1}^\varepsilon, \frac{\alpha}{s}\}\right)^\alpha ds$$
$$ +L_F\|E\Phi_0-\Phi_\varepsilon\|_\infty \int_{-\infty}^0 e^{\lambda_{m+1}^\eps s}\left(\max\{\lambda_{m+1}^\varepsilon, \frac{\alpha}{s}\}\right)^\alpha ds.$$
The second term in the last expression can be estimated with Lemma \ref{integralmaximo}, since
$$\int_{-\infty}^0 e^{(\lambda_{m+1}^\eps-\lambda_m^\eps) s}\left(\max\{\lambda_{m+1}^\varepsilon, \frac{\alpha}{s}\}\right)^\alpha
\leq (1-\alpha)^{-1}(\lambda_{m+1}^\eps)^{\alpha-1}+(\lambda_{m+1}^\eps)^{\alpha}(\lambda_{m+1}^\eps-\lambda_m^\eps)^{-1}$$
which is uniformly bounded as $\eps\to 0$. Then, the second term is bounded by $C(|\bar p^0|+1)\tau(\eps)$ with $C$ a constant independent of $\eps$.  Similar estimate is obtained for the third term: it will be bounded by $C\rho(\eps)$ with $C$ a constant independent of $\eps$. 

For the fourth term 
$$\int_{-\infty}^0 e^{\lambda_{m+1}^\eps s}\left(\max\{\lambda_{m+1}^\varepsilon, \frac{\alpha}{s}\}\right)^\alpha
\leq (1-\alpha)^{-1}(\lambda_{m+1}^\eps)^{\alpha-1}+(\lambda_{m+1}^\eps)^{\alpha-1}\leq 2(1-\alpha)^{-1}(\lambda_{m+1}^\eps)^{\alpha-1}.$$
Which implies that , 

$$ L_F\|E\Phi_0-\Phi_\varepsilon\|_\infty \int_{-\infty}^0 e^{\lambda_{m+1}^\eps s}\left(\max\{\lambda_{m+1}^\varepsilon, \frac{\alpha}{s}\}\right)^\alpha ds\leq 2L_F(1-\alpha)^{-1}(\lambda_{m+1}^\eps)^{\alpha-1}\|E\Phi_0-\Phi_\varepsilon\|_\infty$$

The first term need to be estimated with the aid of Lemma \ref{distanciapes}. Notice that,   
$$4L_F\int_{-\infty}^0 e^{\lambda_{m+1}^\eps s}\left(\max\{\lambda_{m+1}^\varepsilon, \frac{\alpha}{s}\}\right)^\alpha \|p_\eps(s)-Ep_0(s)\|_{X_\eps^\alpha}ds\leq$$
$$\leq \frac{L_F}{3} \|E\Phi_0-\Phi_\varepsilon\|_{\infty} \int_{-\infty}^0 e^{(\lambda_{m+1}^\eps-\lambda_m^\eps-4L_F(\lambda_m^\eps)^\alpha) s}\left(\max\{\lambda_{m+1}^\varepsilon, \frac{\alpha}{s}\}\right)^\alpha ds+$$
$$+4L_F \rho(\eps)  \int_{-\infty}^0 e^{(\lambda_{m+1}^\eps-\lambda_m^\eps-4L_F(\lambda_m^\eps)^\alpha) s}\left(\max\{\lambda_{m+1}^\varepsilon, \frac{\alpha}{s}\}\right)^\alpha ds+$$
$$+4K_2L_F  \tau(\eps) \int_{-\infty}^0 e^{(\lambda_{m+1}^\eps-\lambda_m^\eps-4L_F(\lambda_m^\eps)^\alpha) s}\left(\max\{\lambda_{m+1}^\varepsilon, \frac{\alpha}{s}\}\right)^\alpha (|s|+e^{-2 s})ds$$

With similar arguments as above, the last two terms are bounded by $C\rho(\eps)$  and $C\tau(\eps)$ with $C$ a constant independent of $\eps$. 

The first term is bounded by 
$$\frac{L_F}{3} \|E\Phi_0-\Phi_\varepsilon\|_{\infty}\left((1-\alpha)^{-1}(\lambda_{m+1}^\eps)^{\alpha-1}+\frac{(\lambda_{m+1}^\eps)^\alpha}{ \lambda_{m+1}^\eps-\lambda_m^\eps-4L_F(\lambda_m^\eps)^\alpha}\right)$$

Putting all these estimates together, we have

$$
\|I_1\|_{X_\eps^\alpha}\leq \Big[ 2L_F(1-\alpha)^{-1}(\lambda_{m+1}^\eps)^{\alpha-1}$$
$$+\frac{L_F}{3} \left((1-\alpha)^{-1}(\lambda_{m+1}^\eps)^{\alpha-1}+\frac{(\lambda_{m+1}^\eps)^\alpha}{ \lambda_{m+1}^\eps-\lambda_m^\eps-4L_F(\lambda_m^\eps)^\alpha}\right)\Big]\|E\Phi_0-\Phi_\varepsilon\|_{\infty}+$$
$$+C(|\bar p^0|+1)\tau(\eps)+C\rho(\eps)$$
$$\leq \left(3L_F(1-\alpha)^{-1}(\lambda_{m+1}^\eps)^{\alpha-1}+\frac{L_F(\lambda_{m+1}^\eps)^\alpha}{ \lambda_{m+1}^\eps-\lambda_m^\eps-4L_F(\lambda_m^\eps)^\alpha}\right)\|E\Phi_0-\Phi_\varepsilon\|_{\infty}$$
$$+C(|\bar p^0|+1)\tau(\eps)+C\rho(\eps)\leq \frac{1}{2}\|E\Phi_0-\Phi_\varepsilon\|_{\infty}+C(|\bar p^0|+1)\tau(\eps)+C\rho(\eps)$$
where we have used \eqref{gaps-eps}. 

Now we estimate $I_2$.

$$\|I_2\|_{X_\eps^\alpha}\leq \int_{-\infty}^0 \| \left(e^{A_\varepsilon  s}\mathbf{Q}^\eps_m- Ee^{A_0 s}\mathbf{Q}^0_m\right)\|_{\mathcal{L}(X_0, X^\alpha_\varepsilon)}\|\tilde F_0\|_{X_0}ds$$
$$\leq \int_{-\infty}^0 C_5 e^{-(\lambda_{m+1}^0-1) t}l_\varepsilon^\alpha(t)C_Fdt\leq \frac{2C_5C_F}{1-\alpha}\tau(\eps)|\log(\tau(\eps))|$$
where we have used  Lemma \ref{estimacionsemigruposlinealesproyectados} and Lemma \ref{lepsilon}. 

Putting together the estimates for $I_1$ and $I_2$, we get

$$\|\Phi_\eps(\bar p^0)-E\Phi_0(\bar p^0)\|_{X_\eps^\alpha}\leq \frac{1}{2}\|\Phi_\varepsilon-E\Phi_0\|_{\infty}+C(|\bar p^0|+1)\tau(\eps)$$
$$+C\rho(\eps)+ \frac{2C_5C_F}{1-\alpha}\tau(\eps)|\log(\tau(\eps))|$$

Now since $\Phi_\eps$ and $\Phi_0$ are of compact support, we take the sup norm for $\bar p^0$ with $|\bar p^0|\leq R$, where $R$ is an upper bound of the support of all inertial manifolds and obtain

$$\|\Phi_\varepsilon-E\Phi_0\|_{\infty}\leq \frac{1}{2}\|E\Phi_0-\Phi_\varepsilon\|_{\infty}+C(R+1)\tau(\eps)+C\rho(\eps)+ \frac{2C_5C_F}{1-\alpha}\tau(\eps)|\log(\tau(\eps))|$$
which implies that
$$\|\Phi_\varepsilon-E\Phi_0\|_{\infty}\leq C(\rho(\eps)+ \tau(\eps)|\log(\tau(\eps))|)$$
which shows the theorem. 
\end{proof}

\section{Final Remarks}\label{comment}

\par\bigskip Let us consider now some general remarks about the results of this article.

\begin{re}
We have worked our results in a Hilbert space functional setting, but most of the results, ideas and techniques can be easily adapted to
the more general setting of Banach spaces.  

\end{re}

\begin{re}
Although we have considered the convergence of the inertial manifolds in the ``sup'' norm,  it is possible to analyze the convergence in
stronger norms.  Notice first that for fixed $\eps$, the results from \cite{ChaoBiaoCo,  ChaoLuSell}, see also \cite{Sell&You} guarantee that the inertial manifold  is smooth, as long as the nonlinearity of the equations is smooth enough.  Moroever, with similar arguments as the ones developed in this paper, one could also obtained the convergence of the inertial manifolds in stronger norms, like $C^1$ or $C^{1,\theta}$.  Moreover, some rates could also be obtained for this stronger norm. This is the subject of a future publication.

\end{re}

\begin{re}

As we mentioned in the introduction, one of the motivations for this work is the analysis of a reaction diffusion equation in thin domains, see \cite{Hale&Raugel3, Raugel}.  If we start with a thin $N$-dimensional domain which collapses appropriately into a one dimensional domain, the limit equation is of Sturm-Liouville type and we will be able to apply this result, obtaining rates of the convergence of inertial manifolds, which improve the existing ones in for instance \cite{Hale&Raugel3}. Moreover, we will use the estimates obtained in the present paper to get  better (almost optimal) estimates on the distance of the attractors for thin domains, see \cite{Arrieta-Santamaria-2}. 

\end{re}

\end{document}